\documentclass[10pt,psamsfonts]{amsart}
\usepackage[utf8]{inputenc}

\usepackage{amsmath}
\usepackage{amsthm}
\usepackage{amssymb}
\usepackage{amscd}
\usepackage{amsfonts}
\usepackage{amsbsy}
\usepackage{graphicx}
\usepackage[dvips]{psfrag}
\usepackage{tikz, pgfplots}
\usepackage{array}
\usepackage{color}
\usepackage{xcolor}
\usepackage{epsfig}
\usepackage{url}
\usepackage{hyperref}
\usepackage{float}
\usepackage{stmaryrd}	
\usepackage{overpic}
\usepackage{epstopdf}
\usepackage{epsfig,afterpage}
\usepackage[all]{xy}
\usepackage[makeroom]{cancel}
\usepackage{soul}

\makeatletter
\DeclareFontFamily{U}{tipa}{}
\DeclareFontShape{U}{tipa}{m}{n}{<->tipa10}{}
\newcommand{\ark@char}{{\usefont{U}{tipa}{m}{n}\symbol{62}}}%

\newcommand{\ark}[1]{\mathpalette\ark@arc{$#1$}}

\newcommand{\ark@arc}[2]{%
	\sbox0{$\m@th#1#2$}%
	\vbox{
		\hbox{\resizebox{\wd0}{\height}{\ark@char}}
		\nointerlineskip
		\box0
	}%
}
\makeatother

\newtheorem {theorem} {Theorem} 

\newtheorem {theorema} {Theorem A}

\newtheorem {proposition} {Proposition}

\newtheorem {lemma} {Lemma}
\newtheorem {definition} {Definition}
\newtheorem {remark} {Remark}
\newtheorem {example} {Example}

\newtheorem {corollary} {Corollary}

\begin{document}
\renewcommand{\arraystretch}{1.5}

\title[Dynamics of PSVFs using discrete dynamics]
{some aspects of thermodynamic formalism of piecewise smooth vector fields} 
\author[M.A.C. Florentino,  T. Carvalho and Jeferson Cassiano]
{Marco A. C. Florentino$^1$, Tiago Carvalho$^2$ and Jeferson Cassiano$^3$}

\address{$^1$ DM-UFSCar, Zip Code  13565-905, S\~ao Carlos, S\~ao Paulo, Brazil} \email{marcoflorentino@estudante.ufscar.br}

\address{$^2$ Departamento de Computa\c{c}\~{a}o e Matem\'{a}tica, Faculdade de Filosofia, Ci\^{e}ncias e Letras de Ribeir\~{a}o Preto,
	USP, Av. Bandeirantes, 3900, CEP 14040-901, Ribeir\~{a}o Preto, SP, Brazil.}\email{tiagocarvalho@usp.br}

\address{$^3$ Centro de Matem\'{a}tica Computa\c{c}\~{a}o e Cogni\c{c}\~{a}o. Universidade Federal do ABC, 09210-170. Santo Andr\'{e}. S.P. Brazil} 
\email{jeferson.cassiano@ufabc.edu.br}



\maketitle

\begin{abstract}
	In this paper we study some aspects of thermodynamic formalism, more specifically topological pressure and, as a consequence, topological entropy for piecewise smooth vector fields, using topological conjugation with shift maps and the Perron-Frobenius Operator. Some relationships between entropy and Hausdorff dimensions are also investigated.
\end{abstract}


\section{Introduction}\label{intro}
In $1965$, R. Adler, A. Konheim and M. McAndrew
proposed in \cite{adler} a notion of topological entropy, inspired by the Kolmogorov-Sinai entropy, but whose definition
does not involve any invariant measure. This notion applies to any continuous transformation in a compact topological space. Later, E. Dinaburg (see \cite{din}) and R. Bowen (see \cite{bow1,bow2}) gave a different but equivalent definition for continuous transformations in compact metric spaces. Despite being a bit more restricted, it has the advantage of making the meaning of this concept more transparent: \textit{topological entropy is the exponential growth rate of the number of orbits that are distinguishable within a certain, arbitrarily small degree of precision}.
In addition, Bowen extended the definition to non-compact spaces, which is also very useful in applications.

Topological pressure is a weighted version of topological entropy, where the "weights" are determined by a continuous function called the potential. The idea of pressure was brought from Statistical Mechanics to Ergodic Theory by mathematician and theoretical physicist David Ruelle, one of the originators of the differentiable ergodic theory, and was later extended by the British mathematician Peter Walters.

The study of Piecewise Smooth Vector Fields (PSVFs for short) has been established in recent years not only because of the beauty of the theoretical results but also due to the proximity of this area to applied sciences such as mechanics, engineering, electronics and
biology, as well as social sciences and economics (for applications of PSVFs, see \cite{bernardo} and references therein). It is quite clear that
the Existence and Uniqueness Theorem is not true in the context of PSVFs. On the other hand, under adequate hypotheses, we already know that Poincare Index Theorem, Poincaré-Bendixson Theorem and Peixoto's Theorem have versions for PSVFs (see \cite{bcd,bce,soto}).

In \cite{andre2}, the authors proposed a new way to approach PSVFs through the construction of a metric space of all possible trajectories. Using this, they defined the topological entropy of a planar PSVFs, proved the existence of planar PSVFs with positive entropy (finite and infinite) and provided suficient conditions for a planar PSVF to have infinite entropy. Furthermore, based on this metric space, in \cite{andre1} it is proposed a way to combine the dynamics of a planar PSVF with the shift map in sequence spaces. This approach is absolutely new in the literature and estates tools of discrete dynamics that can be used in order to prove results concerning PSVFs.

In \cite{andre2}, as we said above, the authors introduced the concept of entropy in PSVF and found examples of PSVFs in which entropy is always $\log b$, for a positive integer $b .$ In the present work we obtain examples of PSVFs whose entropy is $\log b$, for a positive \textbf{real} number $b$. In addition, we present the concept of topological pressure for PSVFs. As far as we know, the study of thermodynamic formalism in PSVF is new in the literature and, as we have already said, pressure is a generalization of the concept of entropy. In this work we also use the generalized tent application to calculate the entropy of PSVFs and, as a consequence of this, we obtain the Hausdorff dimension for these vector fields.

\section{Preliminaries}\label{secao teoria basica CVSPs}

Let $V$ be an open of $\mathbb{R}^n$. Consider a manifold $\Sigma \subset V$ of codimension $1 $ in $\mathbb{R}^n$ given by $\Sigma = f^{-1} (0)= \left\lbrace q \in V : f(q) = 0 \right\rbrace $, where $f : V \rightarrow \mathbb{R} $ is smooth having $0 \in \mathbb{R}$ as a regular value (that is, $\nabla f(p) \neq 0$, for $p \in f^{-1} (0)$). We call $\Sigma$ a switching manifold whose boundary separates the regions $\Sigma^{+} = \left\lbrace q \in V : f(q) \geq 0 \right\rbrace $ and $\Sigma^{- } = \left\lbrace q \in V : f(q) \leq 0 \right\rbrace $.

Call $\mathfrak{X}^r$ the space of the $C^r$-vector fields in $V \subset \mathbb{R}^n$ endowed with the
$C^r$-topology, with $r \geq 1$ large enough for our purposes. Call $\mathcal{Z}^r$ the space of PSVFs $Z : V \rightarrow \mathbb{R}^n $ such that

\begin{equation}
	\label{sis}
	Z (q) = \left\{ \begin{array}{c}
		X_{+} (q) \, \, if \,\, q \, \in \Sigma^{+} \\
		X_{-} (q) \, \, if \,\, q \, \in \Sigma^{-}
	\end{array} \right. 
\end{equation}
where $X_{+} = \left( X_{1+}, X_{2+}, \cdots , X_{n+} \right) $, $X_{-} = \left( X_{1-}, X_{2-}, \cdots , X_{n-} \right) \in \mathfrak{X}^ r .$ 

We denote (\ref{sis}) simply by $Z = (X_{+}, X_{-})$ when there is no confusion about the switching manifold. Note that $Z$ is multi-valued in $\Sigma$. We equip $\mathcal{Z}^r$ with the product topology, i.e., $$\parallel Z \parallel_{C^r} = max \left\lbrace \mid X_{+} \mid_{C^r}, \mid X_{-} \mid_{C^r} \right\rbrace , $$ where $\mid \cdot \mid_{C^r}$ denotes the classical $C^r$-norm of the smooth vector fields $X_{+}$ and $X_{-}$ restricted to $\Sigma^{+}$ and $\Sigma^{-}$, respectively.

In order to define rigorously the flow of  $Z$ passing through a point $p \in V$, we distinguish whether this point is at $\Sigma^{\pm} \setminus \Sigma$ or $\Sigma$. For the first two regions, the local trajectory is defined by $X_{+}$ and $X_{-}$ respectively, as usual, but for $\Sigma$ we rely on the contact between the vector fields $X_{+},$ $X_{-}$ and $\Sigma$ characterized by the Lie derivative $X_{+} f(q) = \left\langle \nabla f(q), X_{+}(q) \right\rangle $, where $\left\langle \cdot, \cdot \right\rangle $ is the usual inner product. We also use higher order derivatives given by $X_{+}^k f = X_{+} (X_{+}^{k-1} f) = \left\langle \nabla X_{+}^{k-1} f, X_{+} \right\rangle$.  

A \textbf{Crossing Region } is defined by $\Sigma^{c} = \left\lbrace p \in \Sigma \mid X_{+}f(q) X_{-}f(q) > 0 \right\rbrace $. In addition we denote  $\Sigma^{c^+} = \left\lbrace p \in \Sigma \mid X_{+}f(q) > 0,  X_{-}f(q) > 0 \right\rbrace $ and $\Sigma^{c^-} = \left\lbrace p \in \Sigma \mid X_{+}f(q) < 0,  X_{-}f(q) < 0 \right\rbrace $.
This region $\Sigma^c$ is relatively open in $\Sigma$ and their definition exclude the points where \linebreak $X_{+}f(q) X_{-}f(q) = 0$. These points are on the boundary of this region.

\begin{remark}
In this paper we will not deal with points where $X_{+}f(q) X_{-}f(q) < 0$, where a sliding motion can be considered.
\end{remark}

Any $q \in \Sigma$ such that $X_{+}f(q)X_{-}f(q) = 0$ is called a boundary singularity. The boundary singularities can be of two types: $(i)$ an equilibrium of $X_{+}$ or $X_{-}$ over $\Sigma$ or $(ii)$ a point where a trajectory of $X_{+}$ or $X_{-}$ is tangent to $\Sigma$ (and it is not an equilibrium of  $X_{+}$ or $X_{-}$). In the second case, we call $q \in \Sigma$ a tangential singularity (or tangency
point) and we denote the set of these points by $\Sigma^t$. 

If there exists an orbit of the vector field $X_{+}\mid_{\Sigma^{+}}$ (respectively $X_{-}\mid_{\Sigma^{-}}$) reaching $q \in \Sigma^t$ in a finite time, then such tangency is called a visible tangency for $X_{+}$ (respectively $X_{-}$), otherwise we call $q$ an invisible tangency for $X_{+}$ (respectively $X_{-}$). 

We may also distinguish a particular tangential singularity called two-fold, which is a 
tangency $q$ of $X_{+}$ and $X_{-}$ simultaneously (that is, $X_{+}f(q) = X_{-}f(q) = 0$) satisfying $X_{+}^ 2f(q) X_{-}^2f(q) \neq 0$. A two-fold is called
\begin{enumerate}
\item[$1.$] visible-visible, if it is a visible tangency for both $X_{+}$ and $X_{-}$;
\item[$2.$] invisible-invisible, if it is an invisible tangency for $X_{+}$ and $X_{-}$;
\item[$3.$] visible-invisible, whether it is a visible tangency for $X_{+}$ and an invisible tangency for $X_{-}$ or vice versa.
\end{enumerate} 



%


\textbf{In the sequel of the paper, we will consider just planar PSVFs.} In this case, we say that a tangency point $p \in V$ is singular if $p$ is an invisible-invisible tangency for both $X_{+}$ and $X_{-}$. On the other hand, a tangency point $p \in V$ is regular if it is not singular.

\begin{definition}
	\label{def2}
	The local trajectory (orbit) $\phi_Z (t, p)$ of a PSVF $Z = (X_{+}, X_{-})$ through a small neighborhood of $p \in U$ is defined as follows: \begin{enumerate}
		\item[$(i)$] For $p \in \Sigma^{+}$ and $p \in \Sigma^{-}$ the trajectory is given by $\phi_Z (t, p) = \phi_{X_{+}} ( t,p)$ and $\phi_Z (t,p) = \phi_{X_{-} }(t,p)$ respectively. 
		\item[$(ii)$] For $p \in \Sigma^{c^{+}}$ and taking the origin of time at $p$ the trajectory is defined as $\phi_Z (t, p) = \phi_{X_{-}} (t,p)$ for $t \leq 0$ and $\phi_Z (t,p) = \phi_{X_{+} }(t,p)$ for $t \geq 0.$ If $p \in \Sigma^{c^{-}}$ the definition is the same inverting time;
		\item[$(iii)$] For $p$ a point of regular tangency and taking the origin of time at $p$ the trajectory is defined as $\phi_Z (t, p) = \phi_1 (t,p)$ for $t  \leq 0$ and $\phi_Z (t,p) = \phi_2 (t,p)$ for $t  \geq 0,$ where each $\phi_1, \phi_2$ is either $\phi_{X_{+}}$ or $\phi_{X_{-}}$; 
		\item[$(iv)$] For $p \in V \subset \mathbb{R}^2$ a singular tangency point, $\phi_Z (t, p) = p$ for all $t \in \mathbb {R}.$ 
	\end{enumerate}
\end{definition}


\begin{definition}
\label{def1} A global trajectory is a concatenation of local trajectories. Moreover, a maximal trajectory $\Gamma_z (t, p_0)$ is a global trajectory that can not be extended to any other global trajectories by joining local ones, that is, if $\widetilde{\Gamma_z} $ is a global trajectory containing $\Gamma_z $ then $\Gamma_z = \widetilde{\Gamma_z}$. In this case, we call $I = (\tau^{-}(p_0), \tau^{+} (p_0))$ the maximal interval of the solution $\Gamma_z$.
\end{definition}

\begin{remark}We should note that the maximal interval of the solution may not cover
	the interval $(- \infty , \infty)$, that is, $\tau^{\pm} (p_0)$ could be finite values.\end{remark}

There is a class of piecewise smooth systems satisfying $X_{+}f(p) =X_{-}f(p)$, this class constitutes a well-known class of Filippov systems called \textbf{refractive systems} (see \cite{bmt}). In this paper, we work with refractive systems which, in addition, satisfy $div(X_{\pm}) = 0 $ in $\Sigma^{\pm},$ in other words, preserves the volume measure in $\Sigma^{\pm}$. Therefore, by the Corollary A of \cite{nova}, the PSVFs preserve the volume measure, that is, the Lebesgue measure, here denoted by $med$.


We will now present some definitions about the Perron-Frobeniu theory

\begin{definition}
	Let $(X, d)$ be a compact metric space and $f : X \to X$ be a continuous map. Let $\psi : X \to \mathbb{C}$ be a function. For the
	given $f$ and $\psi$, we can define an operator $\mathcal{L} = \mathcal{L}_{f, \psi}$ as
		\begin{equation*}
		\label{RPF}
		\mathcal{L} \phi(x) = \sum_{x \in f^{-1}(y)} \psi(x) \phi(x)
	\end{equation*}for $\phi$ in a suitable function space on $X$. The operator we just defined is called a \textbf{transfer operator}. If $\psi$ is positive, i.e., $\psi(x) > 0$ for all $x$ in $X$, then the operator is a positive operator, this means that it maps a positive function to a positive function. A positive transfer operator is also called a \textbf{Ruelle-Perron-Frobenius} (\textbf{RPF}) operator.
\end{definition}



A particular case of the transfer operator is given when we take $\psi = e^{\varphi} $. In this work we will use a particular case of Ruelle-Perron-Frobenius operator, taking $\varphi = - \beta \log \mid \det J_{\mu} f \mid $, where $J_{\mu} f $ is the Jacobian of $f$ with respect to the reference measure $\mu$.  So, $\varphi = - \beta \log \mid \det J_{\mu} f \mid $, is called \textbf{geometric potential}.



Based on the analysis of the Frobenius Operator, we present the following

\begin{definition}
	Let $\mathcal{L}_{f, \psi} : \mathcal{C}^{0}(X, \mathbb{C}) \to \mathcal{C}^{0}(X, \mathbb{C}) $ be the RPF operator, where $\phi \mapsto \mathcal{L}_{f, \psi} (\phi) = \underset{t \in f^{-1}(\cdot)}{\sum} \left( \phi {\psi}\right) (t) =  \underset{k}{\sum} \left( \phi {\psi}\right) \circ h_k $, $h_k$ inverses branches of $\phi$ and $\psi = e^{\varphi}$. So, the dual operator, $\mathcal{L}_{f, \psi}^{\ast} : \mathcal{M}(X) \to \mathcal{M}(X)$ ($\mathcal{M}(X)$ the measures set in X) is such that 
	\begin{equation*}
		\int_{X} g d \left( \mathcal{L}_{f, \psi}^{\ast} (\mu)\right) =  \int_{X}   \mathcal{L}_{f, \psi} (g) \, d(\mu) .
	\end{equation*}
\end{definition}

The concept of the generalized Perron-Frobenius operator is the analogue of the transfer matrix method of classical statistical mechanics where the free energy of a spin system with finite-range interaction can be obtained from the largest eigenvalue of a matrix, the so called "transfer matrix". In this work we will adopt the Lebesgue measure as the reference measure associated with the RPF.

%

\begin{definition}
	A nonnegative square matrix $A$ is irreducible if for each ordered pair of indices $i, j$, there exists some $n \geq 0$ such that $a_{ij}^n > 0$. We adopt the convention that for any matrix $A^0 = Id$, and so the $1\times 1$ matrix $[0]$ is irreducible. 
\end{definition}

\begin{theorema}
	\label{tpf}
	\textbf{(Perron-Frobenius Theorem)}	
	Let $A$ be a nonnegative square matrix $A$ of order $k$.
	\begin{enumerate}
		\item[$(i)$] There is a nonnegative eigenvalue $\lambda$ such that no eigenvalue of $A$ has absolute value greater than $\lambda$; 
		\item[$(ii)$] We have $\min_i \left( \sum_{j = 1}^{k} a_{ij}\right) \leq \lambda \leq \max_i \left( \sum_{j = 1}^{k} a_{ij}\right) ; $
		\item[$(iii)$] Corresponding to the envalue $\lambda$ ther is a nonnegative left (row) eigenvalue $u = (u_1, \cdots, u_k)$ and a nonnegative right (column) eigenvector 
		
		\begin{center}
			$v =  \begin{pmatrix}
				v_1 \\
				\vdots \\
				v_k 
				
			\end{pmatrix} ;$
		\end{center}
		
		\item[$(iv)$]	If $A$ is a irreducible then $\lambda$ is a simple eigenvalue and the corresponding eingenvector are strictly positive (i.e., $u_i > 0, v_i > 0$ all $i$);
		\item[$(v)$]	If $A$ is a irreducible then $\lambda$ is the  eingenvalue of $A$ a nonnegative eingenvector. 
	\end{enumerate}
	
\end{theorema}
\begin{proof}
	See \cite{walters}.
\end{proof}

The item $(i)$ of the Theorem (A\ref{tpf}) guarantees the existence of an eigenvalue that is the maximum of the absolute values of the eigenvalues of $A$. This eigenvalue is called the \textit{spectral radius} of $A$ and we denote it  by $\rho(A).$

We naturally define a given flow for PSVF $Z$ as 
\begin{equation*}
	\left. 
	\begin{array}{cc}
		T : \mathbb{R} \times \Omega \to  \Omega\\
		\quad \quad \quad \quad \quad \quad \quad \quad \quad (t,\gamma)\mapsto T(t,\gamma)(\cdot) = \gamma(\cdot + t) ,
	\end{array}
	\right. 
\end{equation*}	
where $\Omega = \left\lbrace \gamma :
\textrm{global trajectory of} \; Z \right\rbrace.$  Then we have the time one  map  $T_1(\gamma) = \gamma(. + 1).$

Based on what we have seen so far, we present the following definition:

\begin{definition}
	\label{pressuredef}
	Let $Z = (X_{+}, X_{-} )$ be a PSVF defined over a compact $2$-dimensional surface $M$ and  $\widehat{\Omega} \subseteq \Omega = \left\lbrace \gamma :
	\textrm{global trajectory of} \, Z \right\rbrace $. We define the topological pressure $P$ of $Z = (X_{+}, X_{-} )$ on $M$, as the topological pressure of the map $T_1$ in $ \widehat{\Omega} \subseteq \Omega$, that is, $P(Z) := P(T_1 \left|_{\widehat{\Omega}}, - \beta \log \mid \det J_{\mu}  {T}_1\mid \right.).$ 
\end{definition}

\begin{remark}
	The classical (in the smooth context) definition of topological pressure and its properties used in this work are found in  Subsection (\ref{press}). Also in the section (\ref{psvf}) we define the time-one map $T_1.$
\end{remark}

\begin{remark}
	As the concept of topological pressure is a generalization of the concept of topological entropy,  Definition (\ref{pressuredef}) agrees with  Definition $4.2$ of entropy for PSVF given in \cite{andre2}.
\end{remark}

\begin{remark}[\textbf{Construction of the "petals"}]\label{constru}
	Before stating the main results of this paper, we will make the following construction: consider the PSVF:
	\begin{equation*}
		\label{deck}
		{Z} (x, y) =  \left\{ \begin{array}{cc} \left.\begin{array}{l} 
				X_{+} (x,y)   \; = \; \left( 1, 1 - x \right)  \; \textrm{for}  \; y \; \geq \; 0  \\
				X_{-} (x, y) \; = \; \left( -1, 1 - x \right)  \; \textrm{for}  \; y \; \leq \; 0, 			
			\end{array} \right.
		\end{array}
		\right.
	\end{equation*}
	where $\Sigma = \{ y = 0 \}.$
	
	They are both symmetric and the point $(1, 0)$ is an invisible-invisible	two-fold, and there is a closed trajectory that goes through the origin $(0, 0)$. Take $k = 3$, consider six rays from the origin (one at each multiple of $\frac{\pi}{3}$). In that way, the	plane is divided into six different regions. Number each region from $1$ to $6$, counter clockwise. 	We define $X_{+}$ in region 1, and $x_{-}$ in region 6. Now, define a vector field in each one of these regions, such that in regions 3 and 5 we have a phase portrait that is the one of $X_{+}$ rotated, and in 	regions 2 and 4 the phase portrait is symmetric to the one of $X_{-}$ (see Figure \ref{casok3}). Now there are	three closed arcs that goes through the origin. That defines a PSVF with six different regions
	such that, apart from three invisible two-folds, and the origin, every other point is either regular	or crossing, and every trajectory that does not goes through the origin is closed	(for more details see \cite{andre1}).
	
	\begin{figure}
		\begin{center}
			\begin{tikzpicture}[scale=0.8]
				\draw[thick, red] (-2,0) .. controls (0,1.1) .. (2,0);
				\draw[thick, red] (-2,0) .. controls (0,-1.1) .. (2,0);
				\draw[thick, lightgray] (-0.5,0) .. controls (0,0.275) .. (0.5,0);
				\draw[thick, lightgray] (-0.5,0) .. controls (0,-0.275) .. (0.5,0);		
				\draw[thick, lightgray] (-1.45,1.0) .. controls (0,1.5) .. (3.0,0);	
				\draw[thick, lightgray] (-1.45,-1.0) .. controls (0,-1.5) .. (3.0,0);
				\draw[thick, lightgray, rotate around={120:(-2.0,0.0)}] (-1.45,1.0) .. controls (0,1.5) .. (3.0,0);	
				\draw[thick, lightgray, rotate around={120:(-2.0,0.0)}] (-1.45,-1.0) .. controls (0,-1.5) .. (3.0,0);		
				\draw[thick, lightgray, rotate around={240:(-2.0,0.0)}] (-1.45,1.0) .. controls (0,1.5) .. (3.0,0);	
				\draw[thick, lightgray, rotate around={240:(-2.0,0.0)}] (-1.45,-1.0) .. controls (0,-1.5) .. (3.0,0);	
				\draw[thick,lightgray, rotate around={120:(-2.0,0.0)}] (-0.5,0) .. controls (0,0.275) .. (0.5,0);
				\draw[thick, lightgray, rotate around={120:(-2.0,0.0)}] (-0.5,0) .. controls (0,-0.275) .. (0.5,0);
				\draw[thick,lightgray, rotate around={240:(-2.0,0.0)}] (-0.5,0) .. controls (0,0.275) .. (0.5,0);
				\draw[thick, lightgray, rotate around={240:(-2.0,0.0)}] (-0.5,0) .. controls (0,-0.275) .. (0.5,0);
				\draw[thick,red, rotate around={120:(-2,0)}] (-2,0) .. controls (0,1.1) .. (2,0);
				\draw[thick, red, rotate around={120:(-2,0)}] (-2,0) .. controls (0,-1.1) .. (2,0);
				\draw[thick, red, rotate around={240:(-2,0)}] (-2,0) .. controls (0,1.1) .. (2,0);
				\draw[thick, red, rotate around={240:(-2,0)}] (-2,0) .. controls (0,-1.1) .. (2,0);	
				\draw[thick,black, loosely dashed] (-6.5,0.0)--(4,0.0);
				\draw[thick,black, loosely dashed, rotate=120] (-4.2,1.74)--(6.0,1.74);
				\draw[thick,black, loosely dashed, rotate=240] (6.6,-1.74)--(-3.58,-1.74);		
				\draw [thick, red,->] (-2,0) to [out=30,in=180]   (0,0.8);
				\draw [thick, red,->] (2,0) to [out=200,in=0]   (0,-0.8);
				\draw [thick, red,->, rotate around={120:(-2,0)}] (-2,0) to [out=30,in=180]   (0,0.8);
				\draw [thick, red,->, rotate around={240:(-2,0)}] (-2,0) to [out=30,in=180]   (0,0.8);	
				\draw [thick, red,->, rotate around={120:(-2,0)}] (2,0) to [out=200,in=0]   (0,-0.8);
				\draw [thick, red,->, rotate around={240:(-2,0)}] (2,0) to [out=200,in=0]   (0,-0.8);			
				\draw (0.1,-0.1) node[ below  left  ] {$ I_1 $};
				\draw (-2.1,2.2) node[ below  left  ] {$ I_2 $};
				\draw (-2.2,-1.5) node[ below  left  ] {$ I_3 $};		
			\end{tikzpicture}
		\end{center}
		\caption{Case $k=3$. Set highlighted in red is invariant for $Z$} \label{casok3}
	\end{figure}
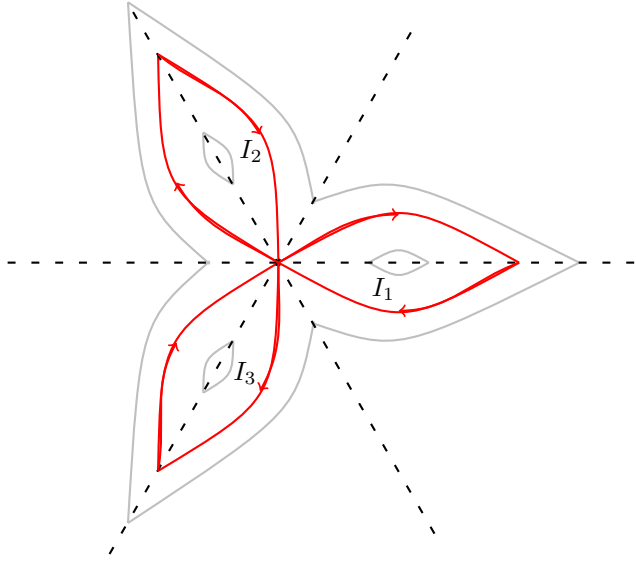
	
	The union of the curves $I_1, I_2, I_3$ highlighted in Figure 1 is invariant to the PSVF and a trajectory in this is the amalgamation of these three different arcs in every possible combination.		
\end{remark}

\section{Main results}

\begin{theorem}
	\label{um}
%
	Given $k \in \mathbb{Z}$, $k \geq 2$ there exists a PSVF $Z$, as in the construction done in Remark  \ref{constru}, with $k$ petals. Then exist  $\widehat{\Omega} \subseteq \Omega = \left\lbrace \gamma :
	\textrm{global trajectory of} \, Z \right\rbrace $ such that $P( T_1\left|_{\widehat{\Omega}} \right. , - \beta \log \mid \det J_{med} {T}_1 \mid) =  
	\log \left( \rho(A)\right)$, where $A$ is an irreducible matrix associated with an oriented graph given by the trajectories of $\widehat{\Omega} .$   
\end{theorem}

\begin{theorem}
	\label{dois}
Consider the following PSVF:
\begin{equation*}
	\label{zk}
{Z_{k}} (x, y) =  \left\{ \begin{array}{cc} \left.\begin{array}{l} 
		X_{{+}_{k}} (x,y)   \; = \; \left( 1, P^{'}_k (x) \right)  \; \textrm{for}  \; y \; \geq \; 0  \\
		X_{{-}_{k}} (x, y) \; = \; \left( -1, P^{'}_k (x) \right)  \; \textrm{for}  \; y \; \leq \; 0,
			\end{array} \right.
\end{array}
\right. 
\end{equation*}where 
\begin{equation*}
	\label{pk}
	P_k(x) = - \left( x + \frac{k-1}{2}\right) \left( x -  \frac{k-1}{2}\right) \prod_{i=1}^{k-1} \left(x - \left( i - \frac{k}{2}\right)  \right)^2 , \quad k \geq 3
\end{equation*}such that $\Sigma = \{ y = 0 \} .$ Then for $k \geq 3$ exist $\widehat{\Omega} \subseteq \overline{\Omega}_k$ satisfying  $P( \overline{T}_1 \left|_{\widehat{\Omega}}\right., - \beta \log \mid \det J_{med} \overline{T}_1 \mid) = \log \left( \rho (A) \right)$, where $A$ is an irreducible  matrix associated with an oriented graph given by the trajectories of  $\widehat{\Omega} \subseteq \overline{\Omega}_k$.

\end{theorem}

\begin{theorem}
	\label{tres}
Consider the following PSVF:
\begin{equation*}
	{Z_{2}} (x, y) =  \left\{ \begin{array}{cc} \left.\begin{array}{l} 
		X_{2_{+}} (x,y)   \; = \; \left( 1, \frac{x}{2} - 4x^3 \right)  \; \textrm{for}  \; y \; \geq \; 0  \\
		X_{2_{-}} (x, y) \; = \; \left( -1, \frac{x}{2} - 4x^3 \right)  \; \textrm{for}  \; y \; \leq \; 0 
		
	\end{array} \right.
\end{array}
\right.
\end{equation*}
 such that $\Sigma = \{ y = 0 \} .$ Then exist $\widehat{\Omega} \subseteq \overline{\Omega}_2$ such that $P( \overline{T}_1 \left|_{\widehat{\Omega}}\right., - \beta \log \mid \det J_{med} \overline{T}_1 \mid) = \log \left( \rho (A) \right)$, where $A$ is an irreducible  matrix associated with an oriented graph given by the trajectories of $\widehat{\Omega} \subseteq \overline{\Omega}_2$.
\end{theorem}

\begin{corollary}
		\label{coro}
	Given $1 < \alpha \leq 2 $, there are trajectories of $Z_2$ such that $h_{top}(Z_2) = \log \alpha .$ Furthermore for each $s \in (0, \log 2]$ there exists a set  $A_s$ such that  $dim_{\mathcal{H}}(A_s) = \log \alpha = h_{top}(Z_2 \left|_{\widehat{\Omega}} \right.) .$
\end{corollary}

\begin{remark}
In the section (\ref{psvf}) we define $\overline{T}_1$, the set $\overline{\Omega}_k, \, k \geq 2$ and explain why we use $\overline{T}_1 $ instead of $T_1$ to calculate topological pressure in Theorems (\ref{dois}) and (\ref{tres}). In addition the definition of Hausdorff Dimension $dim_{\mathcal{H}} (A_{s})$ is given in Subsection [\ref{hausd}] of the Appendix.
\end{remark}


\section{Proof of the main results}\label{secao prova resultados principais}

\subsection{ Proof of Theorem (\ref{um}) }
 \begin{proof}

 Consider a Markov partition $\mathcal{P} = \left\lbrace P_ 0, \cdots, P_{k-1}\right\rbrace $ in the domain of $Z$ such that $\mid \mathcal{P} \mid = k, \; med \left( \mathcal{P}_j \right) > 0  \; \forall j \in \{0,1, \cdots, k-1 \} $  and $$ \phi_Z(\mathcal{P}_j) = \bigcup_{B \subset \mathcal{C} \subset \mathcal{P}_j} B \quad \textrm{for all} \quad \mathcal{P}_j \in \mathcal{P} . $$  So we will have the matrix that represents RPF for the geometric potential  given by :
 
 \begin{equation*}
 	\label{petala}
 		\left.\begin{array}{lll}
 	\left[ \mathcal{L}^{\ast}_{{T}_1 \left|_{\widehat{\Omega}} \right. , - \beta \log \mid \det J_{med} {T}_1 \mid}\right]_{m \times m} = \left[A_{ij} \right]_{m \times m} :  \\ 
 	
 	A_{ij} = \left\{ \begin{array}{c}
 	\left( \frac{med\left( \mathcal{P}_j \cap \phi_{Z}^{-1}\left(\mathcal{P}_i \right) \right) }{med \left(\mathcal{P}_j \right) }\right)^{\beta}, \quad 
 	med\left( \mathcal{P}_j \cap \phi_{Z}^{-1}\left(\mathcal{P}_i \right) \right) \neq 0  \\
 		0, \quad \quad  
 		med\left( \mathcal{P}_j \cap \phi_{Z}^{-1}\left(\mathcal{P}_i \right) \right) = 0 .
 	\end{array}  
 	\right.

 		\end{array} \right.
 \end{equation*}
In this way we get the following matrix:

\begin{equation*}
A = 	\left[ \mathcal{L}^{\ast}_{{T}_1 \left|_{\widehat{\Omega}} \right. , - \beta \log \mid \det J_{med} {T}_1 \mid}\right]_{m \times m} =  \begin{pmatrix}
	p_1^{\beta} & 	p_2^{\beta} & \cdots & 	p_m^{\beta} \\
	p_m^{\beta} & 	p_1^{\beta} & \cdots & 	p_{m-1}^{\beta} \\
	\vdots & \vdots & \ddots & \vdots \\
	p_2^{\beta} & 	p_3^{\beta} & \cdots & 	p_1^{\beta} 
	
\end{pmatrix},
\end{equation*}
such that $ \sum_{k = 1}^{m} p_k = 1 ,$ $p_k$ can be saw as the probability of $\phi_{Z}(x) \in I_k$ if $x \in I_i$,  for all $1 \leq  i \leq m, $ where $m$ is the number of petals (for a reference, the red curve on Figure \ref{casok3} has $3$ petals). Note that each subset of trajectories $ \widehat{\Omega} \subset \Omega$ is associated with a graph and further by Lemma $5.5.1$ of \cite{vries} (see the example \ref{exemp1}), each graph is associated with a subshift of finite type. Therefore, by the Perron-Frobenius Theorem (A \ref{tpf}) and by the Corollary $2.3$ of \cite{cioletti}, we get $P( T_1\left|_{\widehat{\Omega}} \right., - \beta \log \mid \det J_{ med} {T}_1 \mid) = \log \left( \rho(A) \right) $. 

Furthermore if $\beta = 0$, the transfer matrix is an adjacency's matrix and for $ \beta = 1$ this matrix is the stochastic
matrix. As this operator admits a finite representation, the topological pressure is given by \[P(\beta) = \log \rho \left( \mathcal{L}^{\ast}_{- \beta \log \mid \det J_{med} {T}_1 \mid } \right) = \log \rho \left( \mathcal{L}_{- \beta \log \mid \det J_{med} {T}_1 \mid } \right) = \log \left( \rho(A) \right).\]Moreover, 
 if $\beta = 1$, by Perron-Frobenius Theorem (A \ref{tpf}), we get $\rho(A) = 1$ and the topological pressure is zero.

 \end{proof}

\subsection{ Proof of Theorem (\ref{dois}) }
\begin{proof}
	Let $\Omega_k, \, k \geq 3$ be the set of all trajectories contained in $\Lambda_k$ and $s : \Omega_k \to   \left\lbrace  0, 1, 2, \cdots, k \right\rbrace^{\mathbb{Z}}$. If we take $\overline{\Omega}_k = \Omega/s$ and the functions $\overline{s}$ and $\overline{T}_1$ (see (\ref{psvf})) as before, we have that $\overline{s}(\overline{\Omega}_k)$ is a subshift of $\left\lbrace  0, 1, 2, \cdots, k \right\rbrace^{\mathbb{Z}}$ and $\overline{s}$ is a conjugation between $\overline{T}_1$ and the shift $\sigma_k$.  Note that such subshift is associated to the transition matrix:
	\begin{equation*}
		\label{matrixa}
		A = 	\left[ \mathcal{L}^{\ast}_{\overline{T}_1 \left|_{\widehat{\Omega}} \right. , - \beta \log \mid \det J_{med} \overline{T}_1 \mid}\right]_{k \times k} =  \begin{pmatrix}
			p_1^{\beta} & 	(1 - p_1)^{\beta} & 0 & 0 & \cdots & 0 & 0	 \\
			0 & 0 & p_2^{\beta} & (1 - p_2)^{\beta} & \cdots  & 0 & 0 \\
			(1 - p_2)^{\beta} &  p_2^{\beta} & 0 & 0 & \cdots & 0 & 0 \\
			\vdots & \vdots & \vdots & \vdots & \ddots & \vdots & \vdots\\
			0 & 0 & 0 & 0 & \cdots & 	(1 - p_1)^{\beta} &	p_1^{\beta} 
			
		\end{pmatrix} ,
	\end{equation*}

where $p_1$ can be saw as the probability of $\phi_{Z}(x) \in I_0$ if $x \in I_0$ (or tha probability of $\phi_{Z}(x) \in$ $I_k$ if  $x \in I_k$)
 ; and $p_2$ can be saw as the probability of  $ \phi_{Z}(x) \in I_1$ if $x \in I_1$ (or the probability of  $ \phi_{Z}(x) \in I_j, \; j \in \{1, 2, \cdots k-1 \}$ is $ x \in I_j$). 

Note that $A$ also represents a matrix of an oriented graph, where each entry represents $a_{ij}$ the probability that there is a connection (graph edge) between the arc $I_i$ and the arc $I_j$ (vertices of the graph). Furthermore  each oriented graph is conjugated to a finite subshift (see \cite{vries}, Lemma $5.5.1$) which in turn is conjugated to the space of all trajectories of the vector field $Z_k$ (see \cite{andre2}). So, from the Perron-Frobenius Theorem (A \ref{tpf}) and from the Corollary $2.3$ of \cite{cioletti}, it follows that $$P( T_1\left|_{\widehat{\Omega}} \right., - \beta \log \mid \det J_{ med} {T}_1 \mid) = \log \left( \rho(A) \right). $$

\end{proof}

\subsection{ Proof of Theorem (\ref{tres})}

\begin{proof}
Let $\Omega_2$ be the set of all trajectories contained in $\Lambda_2$ and $s : \Omega_2 \to   \left\lbrace  0, 1\right\rbrace^{\mathbb{Z}}$. If we take $\overline{\Omega}_2 = \Omega/s$ and the functions $\overline{s}$ and $\overline{T}_1$ (see (\ref{psvf})) as before, we have that $\overline{s}(\overline{\Omega}_2)$ is a subshift of $\left\lbrace  0, 1\right\rbrace^{\mathbb{Z}}$ and $\overline{s}$ is a conjugation between $\overline{T}_1$ and the shift $\sigma_2$.  See that such subshift is associated to the transition matrix:

\begin{equation*}
	 A = 	\left[ \mathcal{L}^{\ast}_{\overline{T}_1 \left|_{\widehat{\Omega}} \right. , - \beta \log \mid \det J_{med} \overline{T}_1 \mid}\right]_{m \times m} =  \begin{pmatrix}
		p_1^{\beta} & 	p_2^{\beta} \\
		p_2^{\beta} &  	p_1^{\beta} .
		
	\end{pmatrix}
\end{equation*}

Where $p_1$ can be saw as the probability   if $x \in I_0$ and $p_2$ can be saw as the probability   if $x \in I_1$. Again note that $A$ also represents a matrix of an oriented graph, where each entry represents $a_{ij}$ the probability that there is a connection (graph edge) between the arc $I_0$ and the arc $I_1$ (vertices of the graph). Furthermore  each oriented graph is conjugated to the finite subshift of two symbols (see \cite{vries}, Lemma $5.5.1$) which in turn is conjugated to the space of all trajectories of the field $Z_2$ (see \cite{andre2}). So, from the Perron-Frobenius Theorem (A \ref{tpf}) and from the Corollary $2.3$ of \cite{cioletti}, it follows that $$P( T_1\left|_{\widehat{\Omega}} \right., - \beta \log \mid \det J_{ med} {T}_1 \mid) = \log \left( \rho(A) \right). $$

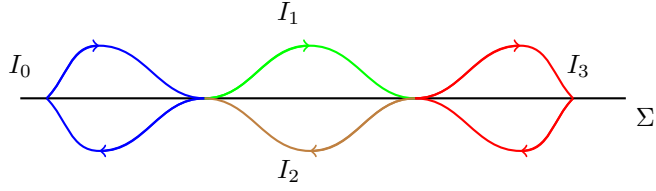
\begin{figure}
	\centering
	\begin{tikzpicture}[scale=0.7]
		\draw [thick, blue,->] (0,0) to [out=45,in=180]   (1,1);
		\draw [thick, blue,->] (3,0) to [out=180,in=0]   (1,-1);			
		\draw [thick, ->, green] (3,0) to [out=0,in=180] (5,1);
		\draw [thick, brown, ->] (7,0) to [out=180,in=0] (5,-1);
		\draw [thick, red, ->, rotate around={180:(5,0)}] (0,0) to [out=45,in=180] (1,1);
		\draw [thick, red,->] (7,0) to [out=0,in=180]   (9,1);	
		\draw[thick,black] (-0.5,0.0)--(11,0.0)node[below  right] {$ \Sigma$};
		\draw (-0.1,1) node[ below  left  ] {$I_0$};
		\draw (5.0,2) node[ below  left  ] {$I_1$};
		\draw (5.0,-1) node[ below  left  ] {$I_2$};
		\draw (10.5,1) node[ below  left  ] {$I_3$};
		\draw [thick, blue] (0,0) to [out=45,in=180] (1,1) to [out=0,in=180] (3,0);
		\draw [thick, blue, rotate around x=180] (0,0) to [out=-45,in=180] (1,1) to [out=0,in=180] (3,0);
		\draw [thick, green] (3,0) to [out=0,in=180] (5,1) to [out=0,in=180] (7,0);
		\draw [thick, brown, rotate around x=180] (3,0) to [out=0,in=180] (5,1) to [out=0,in=180] (7,0);				
		\draw [thick, red, rotate around={180:(5,0)}] (0,0) to [out=45,in=180] (1,1) to [out=0,in=180] (3,0);
		\draw [thick, red, rotate around={180:(5,0)}, rotate around x=180] (0,0) to [out=-45,in=180] (1,1) to [out=0,in=180] (3,0);
\end{tikzpicture}
\caption{Case $k=3$.}
\end{figure}
\end{proof}

\subsection{ Proof of Corollary (\ref{coro})} 

\begin{proof}
Consider the following family of mappings, called the generalized tent map as follows: for $1 < \alpha \leq 2$, let

\begin{equation}
\label{tenda}
\mathcal{T}_{\alpha}(x) := \left\{ \begin{array}{c}
	\alpha x \, \, for \,\, 0 \leq x \leq \frac{1}{2} \\
	\alpha(1 - x) \, \, for \,\, \frac{1}{2} \leq x \leq 1.
	\end{array}  
	\right.
	\end{equation}

%

	It follows from $6.3.14$ of \cite{vries} that for these values of $\alpha$ the system $([0, 1], \mathcal{T}_{\alpha})$ is a factor of some shift system  $(\mathcal{Z}, \sigma_{\mathcal{Z}})$ (for more details see (\ref{tent}) ) 
	over the symbol set $\{0, 1 \}$ under an at most $2$-to-$1$ factor map. So, Theorem $8.2.7$ of \cite{vries} implies that $h_{top}(\mathcal{T}_{\alpha}) =  h_{top}(\sigma_{\mathcal{Z}}) = h_{top}(\sigma_{2}) .$ Therefore, from Section $8.3.4$ of \cite{vries}, it follows that if $1 < \alpha \leq 2$ then $h_{top}(\mathcal{T}_{\alpha}) = \log \alpha .$

	In this way there is a connection between the "world" of shift spaces and the "world" of the generalized tent map. So, from  Theorem (\ref{tres}), it follows that for every $1 < \alpha \leq 2$ we have subsets $\widehat{\Omega} \subseteq \overline{\Omega}_2$ whose entropy $h_{top}(Z_2 \left|_{\widehat{\Omega}} \right.) = h_{top}(T_1 \left|_{\widehat{\Omega}} \right.) = \log \alpha .$
	
	Finally, from Corollary $5.1.18$ of \cite{worth}, for every $s \in (0, \log 2] \subset [0, 1]$ there exists a set $A_s \subset (0, \log 2]$ such that $dim_{\mathcal{H}}(A_s) = s.$ Also, from what we saw in the previous paragraph, for every $s \in (0, \log 2]$, there is $1 < \alpha \leq 2$ such that $s = \log \alpha .$ So it occurs that $$dim_{\mathcal{H}}(A_s)  = \log \alpha = h_{top}(\mathcal{T}_{\alpha}) = h_{top}(Z_2 \left|_{\widehat{\Omega}} \right.) .$$
	\end{proof}
\begin{figure}
	\centering
	\begin{tikzpicture}[scale=0.8]
		\draw [thick, blue,->] (0,0) to [out=45,in=180]   (1,1);
		\draw [thick, blue,->] (3,0) to [out=180,in=0]   (1,-1);
		\draw [thick, red,->] (3,0) to [out=0,in=180]   (5,1);
		\draw [thick, red,->, rotate around={180:(3,0)}] (0,0) to [out=45,in=180] (1,1);
		\draw[thick,black] (-2,0.0)--(8,0.0)node[below  right] {$ \Sigma$};
		\draw (-0.1,1) node[ below  left  ] {$I_0$};
		\draw (7.5,1) node[ below  left  ] {$I_1$};
		\draw [thick, blue] (0,0) to [out=45,in=180] (1,1) to [out=0,in=180] (3,0);
		\draw [thick, blue, rotate around x=180] (0,0) to [out=-45,in=180] (1,1) to [out=0,in=180] (3,0);
		\draw [thick, red, rotate around={180:(3,0)}] (0,0) to [out=45,in=180] (1,1) to [out=0,in=180] (3,0);
		\draw [thick, red, rotate around={180:(3,0)}, rotate around x=180] (0,0) to [out=-45,in=180] (1,1) to [out=0,in=180] (3,0);
		
	\end{tikzpicture}
\end{figure}

\begin{example}
	\label{exemp1}
In Theorem (\ref{dois}), take $\alpha = 4 $. Let $\widehat{\Omega} \subset \Omega$ be  such that \\ $\widehat{\Omega} =  \left\lbrace \textrm{orbits that:} \, \textrm{from} \, I_1 \, \textrm{go to } \, \, I_2, I_3, I_4 \, \textrm{and from} \,I_j \, \textrm{only goes to}\; I_1, \; j = 2, 3 ,4 \right\rbrace$.

	\begin{figure}[!htb]
	\centering
	\begin{tikzpicture}[scale=0.7]
		\draw[thick, red] (-2,0) .. controls (0,1.1) .. (2,0);
		\draw[thick, red] (-2,0) .. controls (0,-1.1) .. (2,0);
		\draw[thick, lightgray] (-0.5,0) .. controls (0,0.275) .. (0.5,0);
		\draw[thick, lightgray] (-0.5,0) .. controls (0,-0.275) .. (0.5,0);
		
		
		\draw[thick, lightgray] (-1.0,1.0) .. controls (0,1.5) .. (3.0,0);	
		\draw[thick, lightgray] (-1.0,-1.0) .. controls (0,-1.5) .. (3.0,0);
		
		\draw[thick, lightgray, rotate around={90:(-2.0,0.0)}] (-1.0,1.0) .. controls (0,1.5) .. (3.0,0);	
		\draw[thick, lightgray, rotate around={90:(-2.0,0.0)}] (-1.0,-1.0) .. controls (0,-1.5) .. (3.0,0);
		
		\draw[thick, lightgray, rotate around={180:(-2.0,0.0)}] (-1.0,1.0) .. controls (0,1.5) .. (3.0,0);	
		\draw[thick, lightgray, rotate around={180:(-2.0,0.0)}] (-1.0,-1.0) .. controls (0,-1.5) .. (3.0,0);
		
		\draw[thick, lightgray, rotate around={270:(-2.0,0.0)}] (-1.0,1.0) .. controls (0,1.5) .. (3.0,0);	
		\draw[thick, lightgray, rotate around={270:(-2.0,0.0)}] (-1.0,-1.0) .. controls (0,-1.5) .. (3.0,0);

		\draw[thick,lightgray, rotate around={90:(-2.0,0.0)}] (-0.5,0) .. controls (0,0.275) .. (0.5,0);
		\draw[thick, lightgray, rotate around={90:(-2.0,0.0)}] (-0.5,0) .. controls (0,-0.275) .. (0.5,0);
		
		\draw[thick,lightgray, rotate around={180:(-2.0,0.0)}] (-0.5,0) .. controls (0,0.275) .. (0.5,0);
		\draw[thick, lightgray, rotate around={180:(-2.0,0.0)}] (-0.5,0) .. controls (0,-0.275) .. (0.5,0);
		
		\draw[thick, lightgray, rotate around={270:(-2.0,0.0)}] (-0.5,0) .. controls (0,0.275) .. (0.5,0);
		\draw[thick, lightgray, rotate around={270:(-2.0,0.0)}] (-0.5,0) .. controls (0,-0.275) .. (0.5,0);
		
		\draw[thick,red, rotate around={90:(-2,0)}] (-2,0) .. controls (0,1.1) .. (2,0);
		\draw[thick, red, rotate around={90:(-2,0)}] (-2,0) .. controls (0,-1.1) .. (2,0);
		\draw[thick, red, rotate around={180:(-2,0)}] (-2,0) .. controls (0,1.1) .. (2,0);
		\draw[thick, red, rotate around={180:(-2,0)}] (-2,0) .. controls (0,-1.1) .. (2,0);
		\draw[thick, red, rotate around={270:(-2,0)}] (-2,0) .. controls (0,1.1) .. (2,0);
		\draw[thick, red, rotate around={270:(-2,0)}] (-2,0) .. controls (0,-1.1) .. (2,0);
		
		
		\draw[thick,black, loosely dashed] (-8,0.0)--(4,0.0);
		\draw[thick,black, loosely dashed, rotate=90] (-6,2.0)--(6.0,2.0);
		\draw[thick,black, loosely dashed, rotate=135] (6,1.4)--(-4,1.4);	
		\draw[thick,black, loosely dashed, rotate=225] (-4.0,-1.4)--(6,-1.4);
		
		
		\draw [thick, red,->] (-2,0) to [out=30,in=180]   (0,0.8);
		\draw [thick, red,->] (2,0) to [out=200,in=0]   (0,-0.8);
		
		\draw [thick, red,->] (-2.0,4.0) to [out=-68,in=80]   (-1.25,1.6);

		\draw [thick, red,->] (-6.0,0.0) to [out=-30,in=180]   (-4.0,-0.8);

		\draw [thick, red,->] (-2,0) to [out=50,in=80]   (-2.75,-1.6);
		
		\draw [thick, red,->] (-2,-4) to [out=55,in=-120]   (-1.26,-2.51);

		\draw [thick, red,->, rotate around={90:(-2,0)}] (-2,0) to [out=30,in=180]   (0,0.8);
		\draw [thick, red,->, rotate around={-180:(-2,0)}, rotate around x=180] (-2,0) to [out=-30,in=180]   (0,0.8);


		\draw (0.1,-0.1) node[ below  left  ] {$ I_1 $};
		\draw (-2.1,2.2) node[ below  left  ] {$ I_2 $};
		\draw (-2.8,0.1) node[ below  left  ] {$ I_3 $};
		\draw (-2.09,-1.3) node[ below  left  ] {$ I_4 $};
		\draw [thick, red,->] (-2,0) to [out=30,in=180]   (0,0.8);
		
		\draw (6,0.0) node[circle, black, draw, fill=red](a1){$I_1$}
		(10, 0) node[circle, black, draw, fill=red](a2){$I_3$}
		(10, 4) node[circle, black, draw, fill=red ](a3){$I_2$}
		(10, -4) node[circle, black, draw, fill=red ](a4){$I_4$};
		
		\draw [->, thick, black] (6.0,1.0) to [out=30,in=180]  (9.2,4.4);
		\draw [->, thick, black] (9.1,4.1) to [out=-180,in=0] (6.7,0.5) ;
		
		\draw [->, thick, black] (6.9,0.0) --  (9.1,0.0);
		\draw [->, thick, black] (9.1,-0.5)  to [out=180,in=0] (6.9,-0.5);
		\draw [->, thick, black] (6.7,-1.0) to [out=-30,in=180]  (9.1,-4.1);
		\draw [->, thick, black] (9,-4.6)  to [out=180,in=-90]  (6.2,-1.0) ;
		\path  (a1) edge [loop left, thick] (a1);
		
	\end{tikzpicture}
\caption{Case  $k = 4$  and the related graph.} \label{casok4}
\end{figure}
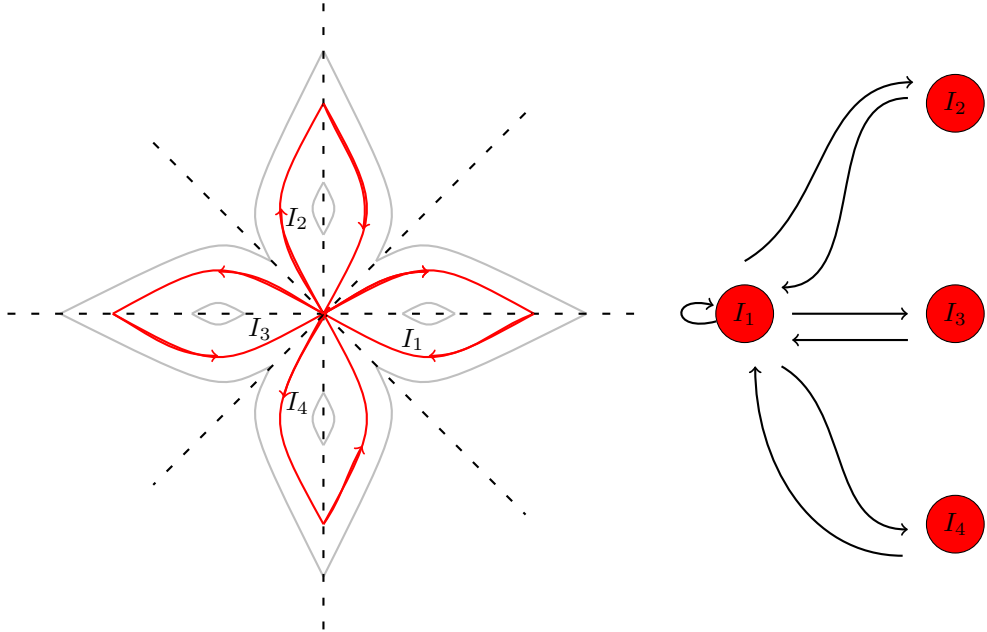

Note that this subset of trajectories is associated with an oriented graph (see Figure \ref{casok4} ), whose transition matrix is as follows:
%

$$A = \left[ \mathcal{L}^{\ast}_{{T}_1 \left|_{\overline{\Omega}_3} \right., 0} \right]_{4\times 4} =
 \begin{pmatrix}
	1 & 1 & 1 & 1 \\
	1 & 0 & 0 & 0 \\
	1 & 0 & 0 & 0 \\
	1 & 0 & 0 & 0 
	
\end{pmatrix}  .$$ 

For $A^2$ we have $a_{ij} > 0$. So, it follows that $A$ is an irreducible matrix  and in addition the eigenvalues of $A$ are $\lambda_1 = 0$ with multiplicity $2$, $\lambda_2 = \frac{1 + \sqrt{13}}{2}$ and $\lambda_3 = \frac{1 - \sqrt {13}}{2}$. Therefore $h_{top}(T_1\left|_{\widehat{\Omega}}\right.) =  \log\left( \rho(A) \right) = \log \left( \frac{1 + \sqrt{13}}{2}\right). $

\end{example}
	
	\begin{remark}
	When $A$ is defined by $1$ in the entire first row and in the entire first column, plus all other entries are null, the shift associated with $A$ is called \textbf{ golden mean shift}.
	\end{remark}

\begin{example}
Consider the case $k = 3$. Then $P_3 (x) = -x^6 + \frac{3x^4}{2} - \frac{9x^2}{15} + \frac{1}{16}$ with roots $\pm 1$ and $\pm \frac{1}{2}$, where the first three have multiplicity $2$ and the last ones are simple. Moreover the PSVF $Z_3$ is: 

\begin{equation}
 {Z_{3}} (x, y) =  \left\{ \begin{array}{cc} \left.\begin{array}{l} 
		X_{{+}_{3}} (x,y)   \; = \; \left( 1, -6x^5 + 6x^3 - \frac{9x}{8} + \frac{9x^2}{2} \right)  \; \textrm{for}  \; y \; \geq \; 0  \\
		X_{{-}_{3}} (x, y) \; = \; \left( -1, -6x^5 + 6x^3 - \frac{9x}{8} + \frac{9x^2}{2} \right)  \; \textrm{for}  \; y \; \leq \; 0.
		
	\end{array} \right.
\end{array}
\right.  
 \end{equation}

The points $p_1 = ( - \frac{1}{2}, 0)$, $p_2 = (\frac{1}{2}, 0)$   are visible-visible fold-folds and the other ones are crossing points  of $Z_3$.

The invariant region is the set: $$\Lambda_3 = \left\lbrace (x, P_3(x)) | - 1 \leq x \leq 1 \right\rbrace \cup  \left\lbrace (x, -P_3(x)) | - 1 \leq x \leq 1 \right\rbrace, $$ that is partitioned into the arcs
$$
I_0 =  \left\lbrace (x, P_3(x)) | - 1 \leq x < - \frac{1}{2} \right\rbrace \cup  \left\lbrace (x, -P_3(x)) | - 1 \leq x < - \frac{1}{2}  \right\rbrace$$ $$I_1 = \left\lbrace (x, P_3(x)) | - \frac{1}{2} \leq x <  \frac{1}{2} \right\rbrace $$ $$   I_2 = \left\lbrace (x, -P_3(x)) |- \frac{1}{2} \leq x <  \frac{1}{2} \right\rbrace $$  $$I_3 =  \left\lbrace (x, P_3(x)) |  \frac{1}{2} \leq x <  1  \right\rbrace \cup  \left\lbrace (x, -P_3(x)) | \frac{1}{2} \leq x <  1  \right\rbrace .
$$as shows Figure \ref{fig caso 3}. 
\end{example}

Now, $\Omega_3$ is the set of all trajectories contained in $\Lambda_3$ and $s : \Omega_3 \rightarrow\left\lbrace 0, 1, 2, 3 \right\rbrace^{\mathbb{Z}}$. If we take $\overline{\Omega}_3 = {\Omega_3}{/} \sim $ and the functions $\overline{s}$ and $\overline{T}_1$ as before, we have that $\overline{s}(\overline{\Omega}_3)$ is a subshift of $\left\lbrace 0, 1, 2, 3 \right\rbrace^{\mathbb{Z}}$ and $\overline{s}$ is a conjugation between $\overline{T}_1$ and the shift $\sigma$. In fact, it is easy to see that such subshift is
associated to the transition matrix:

$$ 	A = \left[ \mathcal{L}^{\ast}_{\overline{T}_1 \left|_{\overline{\Omega}_3} \right., - \beta \log \mid J_{med} \overline{T}_1 \mid}\right]_{4\times 4} =  \begin{pmatrix}
		p_1^{\beta} & 	(1 - p_1)^{\beta} & 0 & 0  \\
		0 & 0 & p_2^{\beta} & (1 - p_2)^{\beta}  \\
		(1 - p_2)^{\beta} &  p_2^{\beta} & 0 & 0  \\
		
		 0 & 0 & 	(1 - p_1)^{\beta} &	p_1^{\beta} 
		
	\end{pmatrix} $$and
$$P\left( \overline{T}_1 \left|_{\overline{\Omega}_3} \right., - \beta \log \mid J_{med} \overline{T}_1 \mid \right) = \log \left(\frac{(p_1^{\beta} + p_2^{\beta}) + \sqrt{(p_1^{\beta} - p_2^{\beta})^2 + 4(1 - p_1)^{\beta}(1 - p_2)^{\beta}}}{2} \right),  $$

where $p_1$ can be saw as the probability of $\phi_{Z}(x) \in I_0$ if $x \in I_0$ (or tha probability of $\phi_{Z}(x) \in$ $I_3$ if  $x \in I_3$); and $p_2$ can be saw as the probability of  $ \phi_{Z}(x) \in I_1$ if $x \in I_1$ (or the probability of  $ \phi_{Z}(x) \in I_2$, if $ x \in I_2$). 

When $\beta = 0$, the matrix $A$ is given by : 

$$ 	A = \left[ \mathcal{L}^{\ast}_{\overline{T}_1 \left|_{\overline{\Omega}_3} \right., 0} \right]_{4\times 4} =  \begin{pmatrix}
	1 & 1 & 0 & 0  \\
	0 & 0 & 1 & 1  \\
	1 &  1 & 0 & 0  \\
	0 & 0 & 1 &	1
	
\end{pmatrix} .$$

Since  $a_{ij} > 0$ for $A^2$, it follows that $A$ is an irreducible matrix  and in addition the eigenvalues of $A$ are $\lambda_1 = 0$ with multiplicity $3$, $\lambda_2 = 2$. Therefore $h_{top}(T_1\left|_{\overline{\Omega}_3}\right.) =  \log\left( \rho(A) \right) = \log \left( 2 \right). $




%
%


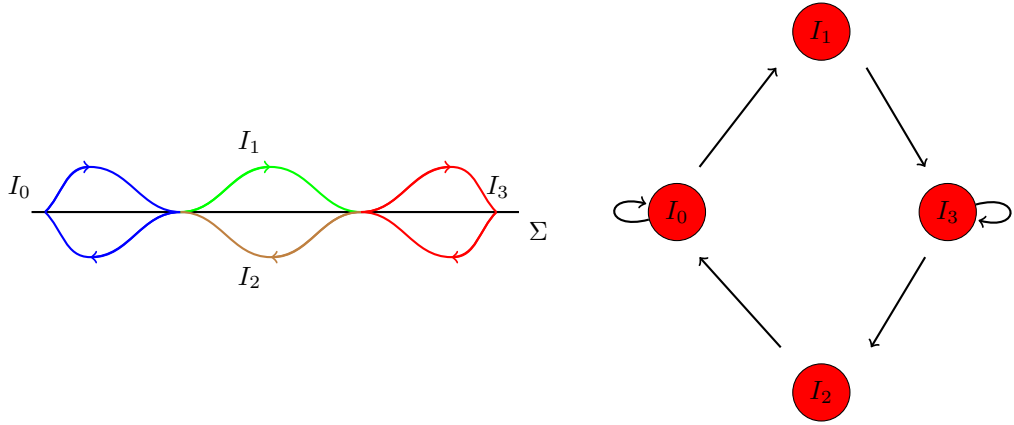
\begin{figure}
	\centering
	\begin{tikzpicture}[scale=0.6]
		\draw [thick, blue,->] (0,0) to [out=45,in=180]   (1,1);
		\draw [thick, blue,->] (3,0) to [out=180,in=0]   (1,-1);			
		\draw [thick, ->, green] (3,0) to [out=0,in=180] (5,1);
		\draw [thick, brown, ->] (7,0) to [out=180,in=0] (5,-1);
		\draw [thick, red, ->, rotate around={180:(5,0)}] (0,0) to [out=45,in=180] (1,1);
		\draw [thick, red,->] (7,0) to [out=0,in=180]   (9,1);	
		\draw[thick,black] (-0.3,0.0)--(10.5,0.0)node[below  right] {$ \Sigma$};
		\draw (-0.1,1) node[ below  left  ] {$I_0$};
		\draw (5.0,2) node[ below  left  ] {$I_1$};
		\draw (5.0,-1) node[ below  left  ] {$I_2$};
		\draw (10.5,1) node[ below  left  ] {$I_3$};
		\draw [thick, blue] (0,0) to [out=45,in=180] (1,1) to [out=0,in=180] (3,0);
		\draw [thick, blue, rotate around x=180] (0,0) to [out=-45,in=180] (1,1) to [out=0,in=180] (3,0);
		\draw [thick, green] (3,0) to [out=0,in=180] (5,1) to [out=0,in=180] (7,0);
		\draw [thick, brown, rotate around x=180] (3,0) to [out=0,in=180] (5,1) to [out=0,in=180] (7,0);				
		\draw [thick, red, rotate around={180:(5,0)}] (0,0) to [out=45,in=180] (1,1) to [out=0,in=180] (3,0);
		\draw [thick, red, rotate around={180:(5,0)}, rotate around x=180] (0,0) to [out=-45,in=180] (1,1) to [out=0,in=180] (3,0);
		\draw (14,0.0) node[circle, black, draw, fill=red](a0){$I_0$}
		(20, 0) node[circle, black, draw, fill=red](a2){$I_3$}
		(17.2, 4) node[circle, black, draw, fill=red ](a1){$I_1$}
		(17.2, -4) node[circle, black, draw, fill=red ](a3){$I_2$};
		
		\draw [->, thick, black] (14.5,1.0) -- (16.2,3.2);
		\draw [->, thick, black] (18.2,3.2) -- (19.5,1.0) ;
		\draw [->, thick, black] (19.5,-1.0) --  (18.3, -3.0);
		\draw [->, thick, black]  (16.3, -3.0) --  (14.5, -1.0);
		\path  (a0) edge [loop left, thick] (a0);
		\path  (a2) edge [loop right, thick] (a2);	
	\end{tikzpicture}
	\caption{Case $k=3$ and the related graph.}\label{fig caso 3}
\end{figure}

\section{Appendix}
\label{apen}
\subsection{Topological Pressure}
\label{press}
\begin{definition}
	Let $\mathfrak{A}$ and $\mathfrak{V}$ be open covers of $X$ compact and let $f : X \to X$ be a continuous transformation. We define their \textbf{join}, as the	collection of all sets of the form ${U} \cup  {V}$, where $U \in \mathfrak{A}$ and $V \in \mathfrak{V}$, and denoted by $ \mathfrak{A} \vee \mathfrak{V}$. Note that this join is a refinement of both covers. This allows us to construct refinements of a single open cover $\mathfrak{A}$. For each $n \in \mathbb{N}$, we define
	
	$$\mathfrak{U}^n := \overset{n-1}{\underset{i = 0}\bigvee}f^{-i}(\mathfrak{U}),$$
	
	where $f^{-i}(\mathfrak{A}) := \{f^{-i}({U}) : U \in \mathfrak{A} \} $
\end{definition}

Let $(X, d)$ be a compact metric space. We will call the elements of $\mathcal{C}(X)$ by \textit{"potentials"}. If $\phi \in \mathcal{C}(X)$, then we will talk about the topological pressure of the potential $\phi$ with respect to $f$. Let $n \in \mathbb{N}$. We denote $\phi^n (x) = \sum_{i = 0}^{n} \phi(f^i(x))$
to the $n$-th Birkoff sum evaluated at a point $x \in X$ for the potential $\phi$.

\begin{definition}
	Let $f$ be a continuous transformation on a compact metric space
	$(X, d)$. Let $ \varphi \in \mathcal{C}(X)$, $n \in \mathbb{N}$ and let $\mathfrak{A}$ be an open cover of $X$. We denote
	
	\begin{equation}
		P_n(f,\phi, \mathfrak{A}) := \inf\left\lbrace \underset {U \in \mathfrak{V}}{\sum} \underset{x \in U}{\sup} \, e^{\phi_n(x)} : \mathfrak{V} \; \textrm{a finite subcover of} \; \mathfrak{U}^{n} \right\rbrace 
	\end{equation}
\end{definition}

Since $\phi$ is bounded in $X$ by compactness, $P_n(f, \phi,  \mathfrak{A})$ is the infimum over a subset of bounded real numbers. Thus $P_n(f, \phi, \mathfrak{A}) < \infty$. We define \textbf{the pressure of the potential $\phi$ with respect to $f$ and the open cover $\mathfrak{A}$} by

\begin{equation}
	P(f, \phi, \mathfrak{A}) := \underset{n \to \infty}{\limsup} \frac{1}{n} \log 	P_n(\phi, f, \mathfrak{A})
\end{equation}

Similar to topological entropy, we want to calculate the pressure as we set the diameter of the cover $\mathfrak{A}$ to zero. The following lemma ensures that this limit exists and does not depend on the choice of covers (see \cite{walters}). 

\begin{lemma}
	Let $\left\lbrace \mathfrak{A}_k \right\rbrace_{k \in \mathbb{N}} $ be any sequence of open covers of $X$ such that 
	
	$$diam(\mathfrak{A}_k) \to 0, \; \textrm{when} \; k \to \infty.$$
	Then the limit $\underset{k \to \infty}{\lim} P(f, \phi, \mathfrak{A}_k)$ exists in $\mathbb{R} \cup \{\infty\}$ and does not depend on the choice of the sequence.
\end{lemma}

\begin{definition}
	Let $f$ be a continuous transformation on a compact metric
	space $(X, d)$. Let $\phi \in \mathcal{C}(X)$ and $\{\mathcal{A}_k\}_{k \in \mathbb{N}}$ be a sequence of open covers such that$diam(\mathfrak{A}_k) \to 0, \; \textrm{when} \; k \to \infty.$. We define \textbf{the topological pressure of the potential $\phi$ with respect to} $f$ as
	
	\begin{equation}
		P(f, \phi) = \underset{k \to \infty}{\lim} P(f, \phi,  \mathfrak{A}_k).
	\end{equation}
\end{definition}

We now introduce the concept of pressure through sets separated by $(n, \epsilon)$ and spanning sets which will be very important throughout this paper

\begin{definition}
	Let $f : X \rightarrow X$ be a continuous map on the space $(X,d)$. A set $E \subset X$ is called $(n,\epsilon)-$separated for $f$ for $n$ a positive integer and $ \epsilon > 0$ provided that for every pair of distinct points $x, y \in E, \,  x \neq y,$ there is at least one $k$ with $0 \leq k < n$
	such that $d(f^k(x), f^k(y)) > \epsilon .$ Another way of expressing this concept is to introduce
	the distance $$d_{n, f}(x,y) = \underset {0 \leq j < n}{\sup}d(f^j(x), f^j(y)).$$
\end{definition}

Using this distance, a set $S \subset X$ is $(n,\epsilon)-$separated for f provided $d_{n,f} (x, y) > \epsilon$ for every pair of distinct points $x, y \in E, \, x \neq y.$

\begin{definition}
	Let $f : X \rightarrow X$ be a continuous map on the space $X$ with metric $d$. Let $K \subset X$ be a subset. For a positive integer $m$, let $$d_{m,f} (w, z) = \underset {0 \leq j < m}{\sup}d(f^j(w), f^j(z))$$ as we defined earlier. A set $E \subset K$ is said to $(n, \epsilon)-$spanning $K$ for $n$ a positive integer and $\epsilon > 0$ provided for each $x \in K$ there exists a $y \in E$ such that    $d_{n,f} (x, y) \leq \epsilon.$
\end{definition}

Thus, given $n \geq 1$ and $\epsilon > 0$, define

\begin{equation}
	G_n(f,\phi, \epsilon) = \inf\left\lbrace \underset {x \in E}{\sum} e^{\phi_n(x)} : \; E \; \textrm{is subset} \; (n, \epsilon)-\textrm{spanning} \; \textrm{of } \; X \right\rbrace \quad \textrm{and}
\end{equation}
\begin{equation}
	S_n(f,\phi, \epsilon) = \inf\left\lbrace \underset {x \in E}{\sum} e^{\phi_n(x)} : \; E \; \textrm{is subset} \; (n, \epsilon)-\textrm{separated} \; \textrm{of } \; X \right\rbrace.
\end{equation}

Then we can define
\begin{equation}
	G(f,\phi, \epsilon) =  \underset {n \to \infty}{\lim} \sup \frac{1}{n} \log G_n(f,\phi, \epsilon) \quad \quad \textrm{and}
\end{equation}
\begin{equation}
	S(f,\phi, \epsilon) =  \underset {n \to \infty}{\lim} \sup \frac{1}{n} \log S_n(f,\phi, \epsilon).
\end{equation}

Also, 
\begin{equation}
	G(f,\phi) =  \underset {\epsilon \to 0}{\lim}  G(f,\phi, \epsilon) \quad \quad \textrm{and}
\end{equation}
\begin{equation}
	S(f,\phi) =  \underset {\epsilon \to 0}{\lim} S(f,\phi, \epsilon).
\end{equation}

\begin{proposition}
	$P(f,\phi) = G(f,\phi) = S(f,\phi)$ for all potential $\phi$ in $X$.
\end{proposition}

\begin{proof}
	See \cite{viana}. 
\end{proof}

Finally, we present some of the properties that the topological pressure satisfies and that will be necessary for the main results (see \cite{walters}). 

\begin{proposition}
	Let $f$ be a continuous transformation on a compact metric space $(X, d$). Let $\phi \in \mathcal{C}(X)$. Then:
	\begin{enumerate}
		\item[$(i)$] $P(0, f) = h_{top}(f);$
		\item[$(i)$] Let $(X_1, d_1)$ and $(X_2, d_2)$ compact metric spaces with continuous maps $f_i: X_i \to X_i$ for $i = 1, 2.$ If $h : X_1 \to X_2$ is a surjective continuous map with
		$h \circ f_1 = f_2 \circ h.$ Then for every $\phi \in \mathcal{C}(X_2)$ we have $P(\phi \circ h, f_1) \geq P(\phi, f_2).$ The equality holds if $h$ is a homeomorphism.	
	\end{enumerate}
	
\end{proposition}

\subsection{The generalized tent map and Shifts}
\label{tent}

Let $\mathcal{A}_k$ be a finite set with at least two elements. Its elements will be called symbols and $\mathcal{A}_k$ will be called the \textit{symbol set} or \textit{alphabet}. The number of elements of $\mathcal{A}_k$ will be denoted by $k$. For convenience, we identify in this subsection $\mathcal{A}_k$ with an initial segment of $\mathbb{Z}^{+}: \mathcal{A}_k = \{ 0, 1, \cdots , k- 1 \}$.

The sequence space $ \mathcal{A}_k^{\mathbb{Z}^{+}}$ is the space of all infinite sequences $x = (x_n)_{n \in \mathbb{Z}^{+}}$ with $x_n \in \mathcal{A}_k$ for all $n \in \mathbb{Z}^{+}.$

The elements of $ \mathcal{A}_k^{\mathbb{Z}^{+}}$ will be denoted without commas and parentheses, as follows: $$x = x_0 x_1 x_2 \cdots x_n \cdots$$ with $x_n \in  \mathcal{A}_k $ for all $n \in \mathbb{Z}^{+}.$ Properly speaking, an element $x \in\mathcal{A}_k^{\mathbb{Z}^{+}}$  is a function $x : \mathbb{Z}^{+} \rightarrow  \mathcal{A}_k;$  its value at an element $n \in \mathbb{Z}^{+}$ is denoted by $x_n$ rather than by the more customary $x(n)$. We call $x_n$ the $n$-th coordinate of $x$ and we also say that the symbol $x_n$ has position $n$ in $x$ or that it occurs at position $n$ The mapping $\pi_n :\mathcal{A}_k^{\mathbb{Z}^{+}} \rightarrow \mathcal{A}_k$ that assigns to every point of $\mathcal{A}_k^{\mathbb{Z}^{+}}$ its $n$-th coordinate will be called the $n$-th projection $(n \in \mathbb{Z}^{+})$. Then $$ x  = \pi_0(x) \pi_1(x) \pi_2(x) \cdots \pi_n(x) \cdots$$ for every $x \in \mathcal{A}_k^{\mathbb{Z}^{+}}$. So if $ \alpha \in \mathcal{A}_k$ and $n \in \mathbb{Z}^{+}$ then $x \in \pi_n^{\leftarrow}[\alpha]$ iff $\alpha$ occurs at position $n$ in $x$,iff $x_n = \alpha$.

A \textit{word} or \textit{block} over $\mathcal{A}_k$ is an element of ${\mathcal{A}_k}^r$ for some $r \in \mathbb{N}$, that is, a finite ordered sequence of elements of $\mathcal{A}_k$, as follows:
$$ b := b_0 \cdots b_{r-1} \in {\mathcal{A}_k}^r = \underbrace{\mathcal{A}_k \times \cdots \times \mathcal{A}_k}_{r \;times}.$$

In this case, the natural number $r$ is called the length of the block $b$. The length of a block $b$ will be denoted by $\mid b \mid$. A block with length $r$ will also be called a $r$-block. There is a unique block with length $0$, the unique member of the set ${\mathcal{A}_k}^0$. It will be called \textit{empty block} or the \textit{empty word} and is denoted by $\boxtimes$. The set of all blocks over $\mathcal{A}_k$ will be denoted by ${\mathcal{A}_k}^{\ast} .$  Stated otherwise, ${\mathcal{A}_k}^{\ast} := \displaystyle\bigcup_{r \in \mathbb{Z}^{+}} {\mathcal{A}_k}^r .$

For blocks we use the notions of "coordinate" and "position" just as for sequences. Thus, if $b = b_0 \cdots b_{\mid b \mid - 1} \in {\mathcal{A}_k}^{\ast}$  then for $n = 0, \cdots , \mid b \mid - 1$ the symbol $b_n$ will be called the $n$-th coordinate of $b$ and we shall say that it is the symbol at position $n$ in $b$.

If $x \in \mathcal{A}_k^{\mathbb{Z}^{+}}$ and $j, l \in \mathbb{Z}^{+}, 0 \leq j < l$, then the block $x_j\cdots x_l$ will be denoted by $x_{[j ; l]}$ and also by $x_{[j ; l+1)}$. The latter notation is particularly convenient for a block with its last position at $l - 1$ for some $l > j: x_{[j ; l)} = x_{[j ; l-1]}.$

New words can be formed by concatenation of two or more words: if $b$ and $c$ are finite words and $b$ or $c$ is empty then $bc = c$ or $b$, respectively; if neither $b$ nor $c$ is empty then $bc := b_0 \cdots b_{\mid b \mid - 1}c_0 \cdots c_{\mid c \mid - 1}$, i.e., it is the word with length $\mid b \mid + \mid c \mid$ whose $n$-th coordinate is
\begin{equation*}
	\label{bloc}
	(bc)_n := \left\{ \begin{array}{c}
		b_n \; \, \, if \,\, 0 \leq n \leq \mid b \mid - 1 \\
		c_{n - \mid b \mid} \, \, if \,\, \mid b \mid \leq n \leq \mid b \mid + \mid c \mid - 1
	\end{array}  
	\right.
\end{equation*}  

Concatenation is easily seen to be associative: if $b$, $c$ and $d$ are words then $(bc)d =
b(cd)$; we shall denote this block by $bcd$. Consequently, concatenation of an arbitrary
finite number of words can be defined and can unambiguously be written without
parentheses.

Let $b$ be a non-empty block. The set of all points of $\mathcal{A}_k^{\mathbb{Z}^{+}}$ that contain the block $b$ in initial position will be called the cylinder on $b$ and it will be denoted by $C_0[b]$. Thus,
\begin{equation*}
	C_0[b] := \left\lbrace x \in \mathcal{A}_k^{\mathbb{Z}^{+}} ; x_{[0; \mid b \mid)} = b \right\rbrace
\end{equation*}
$$= \left\lbrace x \in \mathcal{A}_k^{\mathbb{Z}^{+}}; x_n > b_n \; \textrm{for} \; n = 0, \cdots, \mid b \mid -1 \right\rbrace$$
$$= \left\lbrace \pi_n^{\leftarrow} [b_n] \; n = 0, \cdots, \mid b \mid -1  \right\rbrace.  $$

\begin{remark}
	Notation: $f^{\leftarrow} [A] := \left\lbrace x \in X : f(x) \in A \right\rbrace. $
\end{remark}

We reserve a special notation for cylinders that are based on an initial block of a point
$x \in  \mathcal{A}_k^{\mathbb{Z}^{+}}$, i.e., cylinders of the form $C_0[x_{[0 ; r)}]$ with $r \in \mathbb{N} :$ 
\begin{equation*}
	\widetilde{B}_r(x) := C_0[x_{[0 ; r)}] = \left\lbrace y \in \mathcal{A}_k^{\mathbb{Z}^{+}}; y_n = x_n \; \textrm{for} \: 0 \leq n \leq r-1 \right\rbrace .
\end{equation*}

In this section $\left( X,f \right) $ will always denote an arbitrary dynamical system.
Let $\mathcal{P} := \left\lbrace  P_0, \cdots , P_{k - 1} \right\rbrace $ be a partition of $X$ into $k$ pieces $(k \in \mathbb{N} ).$ We can get an idea of the orbit of a point of $X$ by considering the sequence of pieces that are successively visited. If $x \in X$ then for every $n \in   \mathbb{Z}^{+}$ the point $f^n(x)$ is situated in a unique member $P_{z_n}$ of $P$. In this way we get a sequence $z = z_ 0 z_1 z_2 z_3 \cdots$ of elements from $\mathcal{A}_k := \{ 0, \cdots , k - 1 \}$, i.e., an element of $ \mathcal{A}_k^{\mathbb{Z}^{+}}$. It is called the itinerary of the point $x$ and denoted by $\iota(x)$. So by definition,
\begin{equation}
	\forall \; n \; \in \; \mathbb{Z}^{+} \; : \;  f^n(x) \in P_{\iota(x)_n} .
	\label{itine}
\end{equation}

Also in the case that the sets $P_i$ do not cover $X$ but are still disjoint it is possible that certain
elements of $X$ have an itinerary according to (\ref{itine}) a point $x \in X$ has an itinerary
iff $f^n(x) \in P_0 \cup \cdots \cup P_{k-1}$ for every $n \in \mathbb{Z}^{+}$, iff $x$ belongs to the set
\begin{equation}
	X^{\ast} \left(\mathcal{P}, f \right) := \displaystyle\bigcap_{n \in \mathbb{Z}^{+}} \left( f^n \right)^{\leftarrow} \left[P_0 \cup \cdots \cup P_{k-1} \right]   .
	\label{xast}
\end{equation}

If this set is empty then no point of $X$ has an itinerary. If $x \in X^{\ast} \left(\mathcal{P}, f \right)$ then $f^n(f(x)) = f^{n+1}(x) \in P_{\iota(x)_{n+1}}$ for every $n \in \mathbb{Z}^{+}$, which clearly implies that $f(x) \in X^{\ast} \left(\mathcal{P}, f \right)$ and that the itinerary of $f(x)$ is obtained from the itinerary of $x$ by applying the shift operator to it. So the set $X^{\ast} \left(\mathcal{P}, f \right)$ is invariant under $f$ and on this set the equality $\sigma_k \circ \iota = \iota \circ f$ holds. In particular, it follows that the set of all itineraries of points of $X$ – a subset of $\mathcal{A}_k^{\mathbb{Z}^{+}}$ – is invariant under $\sigma .$

\begin{definition}
	\label{markocpart}
	A \textbf{topological partition} of $X$ is a finite family $\mathcal{P} = \left\lbrace P_0, \cdots ,P_{k-1} \right\rbrace $ of mutually disjoint non-empty open subsets of $X$ whose closures cover $X$:
	
	$$P_i \cap P_j = \emptyset \quad \textrm{for} \quad i \neq j \; \; (i,j = 0,1, \cdots, k-1) \; \; \textrm{and}$$
	$$X = \overline{P_0} \cup \cdots \cup  \overline{P_{k-1}} =  \overline{P_0 \cup \cdots  P_{k-1}} .$$
\end{definition}

So a topological partition has in common with a genuine partition that it consists of
mutually disjoint sets. However, the union of these sets may not be equal to $X$, but it
has to be dense in $X$.

Let $P = \left\lbrace P_0 , \cdots , P_{k - 1}\right\rbrace $ be a topological partition of $X$ and, in accordance with (\ref{itine}) above, consider $ \mathcal{A}_k = \{ 0, \cdots , k - 1 \}$. In addition, let
\begin{equation} 
	\label{dens}
	U_{\mathcal{P}} := P_0 \cup \cdots \cup P_{k-1} .
\end{equation}
Then $U_{\mathcal{P}}$  is a dense open subset of $X$.

An itinerary as defined in (\ref{itine}) with respect to $\mathcal{P}$ will be called a full itinerary. Recapitulating,a point $x \in X$ has a full itinerary $\iota(x) \in \mathcal{A}_k^{\mathbb{Z}^{+}}$ whenever $f^n(x) \in U_{\mathcal{P}}$ for every $n \in \mathbb{Z}^{+}$, in which case $\iota(x)$ is characterized by condition (\ref{itine}). As in (\ref{itine}), the set of points having a full itinerary is denoted by $X^{\ast} \left(\mathcal{P}, f \right)$, so $X^{\ast} \left(\mathcal{P}, f \right) = \displaystyle\bigcap_{n = 0}^{\infty} \left( f^n \right)^{\leftarrow} \left[U_{\mathcal{P}}\right]$ (see also formula (\ref{xast})). Note that $X^{\ast} \left(\mathcal{P}, f \right)$ is the intersection of countably many open sets: a $G_{\delta}$-set. When $P$ and $f$ are understood the set $X^{\ast} \left(\mathcal{P}, f \right)$ will be denoted simply by $X^{\ast}$. If $x \in X^{\ast}$ then its full itinerary $\iota(x)$ is the element of $\mathcal{A}_k^{\mathbb{Z}^{+}}$ characterized by formula (\ref{itine}), which can be rewritten as 
\begin{equation}
	\label{iti}
	x \in \displaystyle\bigcap_{n = 0}^{\infty} \left( f^n \right)^{\leftarrow} \left[P_{\iota(x)_n}\right].
\end{equation}

Apart from full itineraries we shall also consider partial itineraries. A finite block
$b = b_0 \cdots b_{r-1} $ over  $\mathcal{A}_k^{\mathbb{Z}^{+}}$ is said to be a partial itinerary of a point $x \in X$ whenever $f^n(x) \in P_{b_n}$ for $0 \leq n \leq r - 1.$ Thus, if for every $r \in \mathbb{N}$ and every block $b$ of length $r$ we define
\begin{equation*}
	\label{itio}
	D_r(b) := \displaystyle\bigcap_{n = 0}^{r-1} \left( f^n \right)^{\leftarrow} \left[ P_{b_n} \right].
\end{equation*}

then the block $b$ is a partial itinerary of the point $x \in X$ iff $x \in D_r(b)$. In that case, the
set $D_r(b)$ is an open neighbourhood of the point x.

One more notational convention: if $z \in \mathcal{A}_k^{\mathbb{Z}^{+}}$ and $r \in \mathbb{N}$ then the clumsy expression $D_r(z_{[0 ; r)})$ will be simplified to $D_r(z)$; so
\begin{equation*}
	\label{itiou}
	D_r(z) := D_r(z_{[0 ; r}) = \displaystyle\bigcap_{n = 0}^{r-1} \left( f^n \right)^{\leftarrow} \left[ P_{z_n} \right].
\end{equation*}

So, if $x \in X^{\ast}$ then every initial block of the full itinerary $\iota(x)$ is a partial itinerary of $x$. In fact, formula (\ref{iti}) implies that for all $r \in \mathbb{N} $ we get
\begin{equation}
	x \in  \displaystyle\bigcap_{n = 0}^{\infty} \left( f^n \right)^{\leftarrow} \left[P_{\iota(x)_n}\right] \subseteq 
\displaystyle\bigcap_{n = 0}^{r-1} \left( f^n \right)^{\leftarrow} \left[P_{\iota(x)_n}\right] = D_r(\iota(x)) .\end{equation}

For any topological partition $P = \left\lbrace P_0, \cdots , P_{r-1} \right\rbrace$  of $X$ we define in the following way a shift space over the symbol set $\mathcal{A}_k$. Call a block $b$ over $\mathcal{A}_k$ of length $r \geq 1$ forbidden with respect to the pair $(\mathcal{P}, f)$, or just $(\mathcal{P}, f)$-forbidden, whenever $D_r(b) = \emptyset.$ The set of $(\mathcal{P}, f)$-forbidden blocks will be denoted by $\mathcal{B}_{\mathcal{P}, f}$. In accordance with $5.3.5$ of \cite{vries}, the elements
of the set $\mathcal{A}_k^{\ast} \setminus \mathcal{B}_{\mathcal{P}, f}$ are called the $(\mathcal{P}, f)$-allowed words or blocks. Thus, a finite block $b$ is $(\mathcal{P}, f)$-allowed iff $D_r(b)  \neq \emptyset$, iff there is a point $x \in X$ such that $f^n(x) \in P_{b_n}$ for $0 \leq n \leq r - 1$, iff the block $b$ is a partial itinerary of some point $x$ of $X$.

From Section $5.3$ of \cite{vries}, the set $\mathcal{B}_{\mathcal{P}, f}$ defines a subset $ \mathfrak{X}(\mathcal{B}_{\mathcal{P}, f})$ of $\mathcal{A}_k^{\mathbb{Z}^{+}}$, which is a shift space if it is not empty. If this is the case we say that the topological partition $\mathcal{P}$ is $f$-adapted and we call this shift space the symbolic model of $(X,f)$, generated by $\mathcal{P}$. It will be denoted by $\mathcal{Z}(\mathcal{P}, f).$

It follows from the definitions – see also formulas $(5.3-6)$ of \cite{vries} – that a point $z$ of $\mathcal{A}_k^{\mathbb{Z}^{+}}$ belongs to $\mathcal{Z}(\mathcal{P}, f)$ iff every block occurring in $z$ is $(\mathcal{P}, f)$-allowed, iff every block occurring in $z$ is a partial itinerary, iff every initial block of $z$ is a partial itinerary, iff $D_r(z) \neq 0$ for every $r \in \mathbb{N}.$

From now on we assume that  $(X,f)$ is a dynamical system with a \textit{compact Hausdorff} phase space $X$. If $\mathcal{P} = \left\lbrace P_0, P_1, \cdots, P_{k-1}\right\rbrace $ is an $f$-adapted topological partition of $X$ then the shift space $\mathcal{Z}(\mathcal{P}, f)$ exists. Recall that, by definition, for every point $z \in \mathcal{Z}(\mathcal{P}, f)$ and for every $r \in \mathbb{N}$ the initial block $z[0 ; r)$ of $z$ is $(\mathcal{P}, f)$-allowed, which means that the set $D_r(z)$ is non-empty. Consequently, for every point $z$ of $\mathcal{Z}(\mathcal{P}, f)$ the sets $D_r(z)$ for $r = 1, 2, 3, \cdots$ form a decreasing sequence of non-empty closed sets in the compact space $X$, hence they have a non-empty intersection.

The topological partition $\mathcal{P}$ of $X$ is called a \textit{pseudo-Markov} partition whenever it
is $f$-adapted and for every point $z \in \mathcal{Z}(\mathcal{P}, f)$ the set  $\displaystyle\bigcap_{r = 1}^{\infty}D_r(z)$ consists of just one point. If the shift space $\mathcal{Z}(\mathcal{P}, f)$ is, in addition, of finite type then the topological partition $\mathcal{P}$ is called a \textit{Markov partition}.

If $\mathcal{P}$ is a pseudo-Markov partition then for every $z \in \mathcal{Z}(\mathcal{P}, f)$ the unique point of the intersection $\displaystyle\bigcap_{r = 1}^{\infty}D_r(z)$ will be denoted by $\Psi_{\mathcal{P}, f}(z)$. In this way we obtain a mapping $\Psi_{\mathcal{P}, f} :  \mathcal{Z}(\mathcal{P}, f)\rightarrow X$. So by definition, we have 
$$ \forall  \; z \in  \mathcal{Z}(\mathcal{P}, f) : \{ \Psi_{\mathcal{P}, f} \} = \displaystyle\bigcap_{r = 1}^{\infty} \overline{D_r(z)} .$$

Next, we will denote the displacement space $\mathcal{Z}(\mathcal{P}, f)$ only by $\mathcal{Z}$ and the mapping $\Psi_{\mathcal{P}, f} $ by $\Psi$. The Proposition $6.1.8$ of \cite{vries}, shows that 
$\Psi$ is a dynamical systems morphism, so the subset $\Psi[\mathcal{Z}]$ of $X$ is closed ($\mathcal{Z}$ is compact) and invariant, defining a subsystem of $(X, f)$. If $\Psi$ is a surjection we call the morphism $ \Psi : (\mathcal{Z}, \sigma_{\mathcal{Z}}) \rightarrow (X ,f)$ a symbolic representation of the dynamical system $(X ,f)$. Also in the case that $\Psi$ is not surjective, or that we do not yet know it to be surjective, we (sloppily) call the system $(\mathcal{Z}, \sigma_{\mathcal{Z}})$ a symbolic representation of $(X, f).$

\subsection{PSVFs and Shifts}
\label{psvf}

Here we will  identify $\mathcal{A}_k = \{ 0, 1, \cdots , k- 1 \}$.

Given a flow,  $\varphi (t, x)$ of a vector field $W$ defined in an open set $\mathcal{U}$, define $T_1 : \mathcal{U} \rightarrow \mathcal{U}$ the time-one map given by $T_1(x) = \varphi(1, x) $. 

Consider the following PSVFs: 
\begin{equation*}
	{Z_{2}} (x, y) =  \left\{ \begin{array}{cc} \left.\begin{array}{l} 
			X_{2_{+}} (x,y)   \; = \; \left( 1, \frac{x}{2} - 4x^3 \right)  \; \textrm{for}  \; y \; \geq \; 0  \\
			X_{2_{-}} (x, y) \; = \; \left( -1, \frac{x}{2} - 4x^3 \right)  \; \textrm{for}  \; y \; \leq \; 0 
			
		\end{array} \right.
	\end{array}
	\right.
\end{equation*}	

\begin{equation*}
 {Z_{k}} (x, y) =  \left\{ \begin{array}{cc} \left.\begin{array}{l} 
		X_{{+}_{k}} (x,y)   \; = \; \left( 1, P^{'}_k (x) \right)  \; \textrm{for}  \; y \; \geq \; 0  \\
		X_{{-}_{k}} (x, y) \; = \; \left( -1, P^{'}_k (x) \right)  \; \textrm{for}  \; y \; \leq \; 0 
		
	\end{array} \right.
\end{array}
\right. 
\end{equation*}
   where 
$$P_k(x) = - \left( x + \frac{k-1}{2}\right) \left( x -  \frac{k-1}{2}\right) \prod_{i=1}^{k-1} \left(x - \left( i - \frac{k}{2}\right)  \right)^2 . $$ 

Notice that $P_k$ has $2k$ roots, being $2$ simple roots at $r_0 = \frac{1-k}{2}$ and $r_1 = \frac{k-1}
{2}$ and $k-1$ roots of multiplicity two at $p_j = j- \frac{k}{2}$ for $j = 1, \cdots , k -1.$ Moreover, $P^{'}_k (r_0) > 0$, $P^{'}_k (r_1) < 0$, $P^{'}_k (p_j) = 0$,  and $P^{''}_k (p_j) > 0$, for every $j = 1, \cdots , k-1.$  In addition, $(r_0, 0)$ and $(r_1, 0)$ are crossing points of $Z_k$, the points $(p_j, 0)$ are visible-visible two folds of $Z_k$,  $j = 1, 2, \cdots k-1 .$ These vector fields were stated in  \cite{andre1}. 

For each $k < \infty$, consider $$\gamma^{X}_k = \left\lbrace (x, P_k (x)) \mid x \in [r_0, r_1] \right\rbrace \; \textrm{and} \; \gamma^{Y}_k \left\lbrace (x, - P_k (x)) \mid x \in [r_0, r_1] \right\rbrace .$$ Define $\Lambda_k = \gamma^{X}_k \cup \gamma^{Y}_k, $ and  note that $\Lambda_k $  is an invariant set for $Z_k .$

Remember that a PSVF does not have a unique trajectory passing through a point.  Thus the map $T_1 : \Lambda_k \rightarrow \Lambda_k$, $T_1 (x) = \varphi (1, x)$, where $\varphi$ is a flow passing through $x \in \Lambda_k$. 
In order to avoid such a problem we will redefine the function $T_1$, but some new definitions and results must to be considered.

Take the set  $$\Omega_k = \left\lbrace \gamma \; | \; \gamma \; \textrm{is a global trajectory of } Z_k \mbox{ with } \gamma (0) \in \Lambda_k \right\rbrace.$$


%
%

\begin{proposition}
	\label{pro}
	For any $k < \infty$ let $\gamma \in \Omega_k$ be a global trajectory, then for all $t \in \mathbb{R}$, there exists an unique $t^{\ast} \in [t, t+1)$ such that $\gamma (t^{\ast}) \in \lbrace (p_j , 0) \mid j = 1, \cdots, k-1 \rbrace$.
	
	The proof of the previous proposition is found in the reference  \cite{andre1}.
	
\end{proposition}

The region $\Lambda_k$ can be partitioned into arcs that goes from $p_j$ to the adjacent ones ($p_{j+1}$ and $p_{j-1}$) or to itself. So, consider $k$ to be fixed and let 

\begin{center}
	$I_0 = \left\lbrace (x, P_k (x)), \; x \in [r_0, p_1)\right\rbrace \cup \left\lbrace (x, - P_k (x)), \; x \in [r_0, p_1)\right\rbrace, $
\end{center}the arc from $p_1$ to itself passing through $r_0$. For any $j = 1, \cdots , k - 2$, let
$$I_{2j-1} = \left\lbrace (x, P_k (x)), \; x \in (p_j, p_{j+1})\right\rbrace \textrm{e} \; \; I_{2j} = \left\lbrace (x, -P_k (x)), \; x \in (p_j, p_{j+1})\right\rbrace , $$the arcs from $p_j$ to $p_{j+1}$ and from $p_{j+1}$ to $p_j$, respectively. And, $$I_{2k-3} = \left\lbrace (x, P_k (x)), \; x \in (p_{k-1}, r_1] \right\rbrace \cup \left\lbrace (x, - P_k (x)), \; x \in (p_{k-1}, r_1]\right\rbrace.$$In short, we enumerate these arcs top to bottom, left to right. See Figure $1$.

Consider a set $\mathcal{A}_k$ with $k$ elements (say $\mathcal{A}_k = \{0, 1, \cdots , k - 1 \}$) with the discrete topology. Now, consider $\mathcal{A}_k^{\mathbb{Z}}$ , i.e., all the sequences $x = (x_j)_{j\in \mathbb{ Z}}$ , with $x_j \in \mathcal{A}_k$, for all $j$ and the product
topology of all discrete topologies.

\begin{definition}
	\label{def5}
	Let $x = (x_j)_{j \in \mathbb{Z}}$ and $y = (y_j)_{j \in \mathbb{Z}}$ two elements of $\mathcal{A}_k^{\mathbb{Z}}$ . Define $d : \mathcal{A}_k^{\mathbb{Z}} \times \mathcal{A}_k^{\mathbb{Z}} \rightarrow \mathbb{R}$ by:  $$d(x, y) = \underset {i \in \mathbb{Z}}{\sum} \frac{\mid x_i - y_i \mid}{2^ {|i|} }.$$
\end{definition}

\begin{definition}
	\label{def6}
	Define  $ \sigma : \mathcal{A}_k^{\mathbb{Z}} \rightarrow \mathcal{A}_k^{\mathbb{Z}} $ given by $\sigma ((a_j)) = b_j$, where $b_j = a_{j+1}$. The map  is called two-sided full shift and the discrete flow $(\mathcal{A}_k^{\mathbb{Z}}, \sigma)$ is called symbolic flow or shift
	system.
\end{definition}


\begin{definition} 
	\label{def7}
	Let $s : \Omega_k \rightarrow \mathcal{A}_{2(k-1)}^{\mathbb{Z}}$ be given by $s(\gamma) = (s_j(\gamma))_{j \in \mathbb{Z} }$, where : 
	
	$$ s_j(\gamma) =  \left\{ \begin{array}{cc} \left.\begin{array}{l} 
			n \; \textrm{if}  \; \gamma(j) \; \in \; I_n  \\
			m \; \textrm{if}  \; \gamma(j) \; \in \; \lbrace (p_l, 0) \, | l = 1, \cdots , k-1 \rbrace \; \textrm{and} \; \gamma \left( j + \frac{1}{2}\right) \in  I_m.
			
		\end{array} \right.
	\end{array}
	\right.   $$ 
	
	The sequence $s(\gamma)$ is called the itinerary of $\gamma$.
\end{definition}

According to Definition (\ref{def7}), given $\gamma \in \Omega_k$, there exist infinitely many distinct trajectories with the same itinerary of $\gamma$, simply because the initial conditions belong to the same arc $I_n$. In order to avoid this problem we will consider the following definition:

\begin{definition}
	\label{def8}
	Let $\gamma_1 , \gamma_2 \in \Omega_k .$ We say that $\gamma_1 \sim \gamma_2$ if and only if $s(\gamma_1) = s(\gamma_2) .$ Denote $\overline{\Omega}_k = \Omega_k / \sim .$
\end{definition}  

The relation in Definition (\ref{def8}) is an equivalence relation. In fact, for  $\gamma_1 \in \Omega_k$, we have $\gamma_1 \sim \gamma_1$ because $s(\gamma_1) = s(\gamma_1)$. Also for all $\gamma_1, \gamma_2 \in \Omega_k$, if $\gamma_1 \sim \gamma_2$, we have $s(\gamma_1) = s(\gamma_2)$, and in this way, $\gamma_2 \sim \gamma_1$. And finally for all  $\gamma_1, \gamma_2, \gamma_3 \in \Omega_k$, if $\gamma_1 \sim \gamma_2$, we have $s(\gamma_1) = s(\gamma_2)$ and if $\gamma_2 \sim \gamma_3$, we have $s(\gamma_2) = s(\gamma_3)$, and in this way, $ s(\gamma_1) = s(\gamma_2) = s(\gamma_3)$, and therefore $\gamma_1 \sim \gamma_3 .$


\begin{remark}\label{obs representante}
	Observe that, given $\overline{\gamma} \in \overline{\Omega}_k$, there exists a representative $\gamma^{\ast}$ such that  $\gamma^{\ast} (0) \in \lbrace (p_j, 0), j = 1, 2, \cdots , k-1 \rbrace$ 
	(see \cite{andre1}).
\end{remark}

\begin{definition}
	\label{def9}
	The Hausdorff distance between the sets A and B is given by $$d_H(A, B) =  max \left\lbrace \underset{x \in A}{sup} \, \underset{y \in B}{inf} \, d(x,y), \underset{y \in B}{sup} \, \underset{x \in A}{inf} \, d(x,y)  \right\rbrace = max \left\lbrace \underset{x \in A}{sup} \, d(x,B), \underset{y \in B}{sup} \, d(y,A) \right\rbrace. $$
\end{definition}

\begin{definition}
	\label{def10}
	Now, define $\rho_k : \overline{\Omega}_k \times \overline{\Omega}_k \rightarrow \mathbb{R}$, by $$\rho_k (\overline{\gamma_1}, \overline{\gamma_2}) =  \underset {i \in \mathbb{Z}}{\sum} \frac{d_i(\overline{\gamma_1}, \overline{\gamma_2})}{2^{\mid i \mid}},$$where $d_i(\overline{\gamma_1}, \overline{\gamma_2}) = d_H \left( \gamma_1^{\ast}([i, i+1]),  \gamma_2^{\ast}([i, i+1])\right), $ $d_H$ is the Hausdorff distance and $\gamma_1^{\ast} , \gamma_2^{\ast}$ are those representatives given in Remark (\ref{obs representante}).
\end{definition}

For simplicity of notation, hereafter we will refer only to $\gamma \in \overline{\Omega}_k$, meaning the equivalence class $\overline{\gamma}$ with the representative $\gamma^{\ast} \in \lbrace (p_j, 0) \mid j = 1, 2, \cdots , k-1 \rbrace .$

Let $\overline{T_1} : \overline{\Omega}_k \rightarrow \overline{\Omega}_k$ be the function induced by $T_1$, that is, $\overline{T_1}(\overline{\gamma}) = \overline{T_1(\gamma)} .$ Note that the induced function does not depends on the representative. In fact if $s(\gamma_1) = s(\gamma_2) = (s_j)_{j \in \mathbb{Z}}$, then, for all $j \in \mathbb{Z}$ \\

$\left.\begin{array}{ll}
	\gamma_1 (j), \gamma_2(j) \in I_{s_j} \Rightarrow \gamma_1 (j+1), \gamma_2 (j+1) \in I_{s_{j+1}} \Rightarrow \\ \\ T_1 (\gamma_1)(j), T_2 (\gamma_2)(j) \in I_{s_{j+1}} \Rightarrow s(T_1(\gamma_1)) = s(T_1(\gamma_2)) .  
	
\end{array} \right. $ \\

Now let $\overline{s} : \overline{\Omega}_k \rightarrow \lbrace0, 1, \cdots , 2k-3\rbrace^{\mathbb{Z}}$ be the function induced by $s$, that is, $\overline{s}(\gamma) = s (\gamma)$. Note that the induced function does not depend on the representative, because of the equivalence relation and because it is
one-to-one (see \cite{andre1} ).

\subsection{Carath\'{e}odory Construction, Hausdorff Measure and Hausdorff Dimension}
\label{hausd}
The Carath\'{e}odory construction of outer-measures is a general framework with which one can construct many of the standard geometric outer-measures including the Hausdorff measures.
\begin{definition}
	Let $(X, d)$ be a metric space,  $\mathfrak{F} \subseteq \mathcal{P}(X)$ and $\zeta : \mathfrak{F} \rightarrow [0, \infty) $ (potentially Hausdorff) such that 
	\begin{enumerate}
		\item[$(i)$] For all $\delta > 0$ there exist $\left\lbrace  \mathfrak{A}_i \right\rbrace \subset \mathfrak{F} $ such that $X \subset \cup_i \mathfrak{A}_i$ and $d( \mathfrak{A}_i) \leq \delta .$ 
		\item[$(ii)$] For all $\delta > 0$ there exist $ \mathfrak{A}  \in \mathfrak{F} $ such that $\zeta (U) \leq \delta.$ 
	\end{enumerate}
For $\delta > 0$ we define 

\begin{center}
	$\psi_{\delta} : \mathcal{P}(X) \rightarrow [0, \infty] $
\end{center}

$$\psi_{\delta} (A) = inf \left\lbrace \Sigma_i \zeta (\mathfrak{A}_i) : A \subset \cup_i \mathfrak{A}_i, \, d(\mathfrak{A}_i) < \delta \left\lbrace \mathfrak{A}_i\right\rbrace \subset \mathfrak{F} \right\rbrace  $$
\end{definition}

By a $\delta$-cover in the context of the Carath\'{e}odory Construction we mean a countable collection of sets $\{\mathfrak{A}_i\} \subset \mathfrak{F}$ such that $\zeta(\mathfrak{A}_i) \leq \delta$ and $d(\mathfrak{A}_i) \leq \delta$. This definition is dependent on $\mathfrak{F}$, if this is ambiguous we will refer to such covers as
$(\mathfrak{F}, \delta)$-covers. For brevity we write $\psi_{\delta}(A) = \inf \underset{i}{\sum} \zeta(\mathfrak{A}_i)$ where $\{ \mathfrak{A}_i \}$ is understood to be a$(\mathfrak{F}, \delta)$-cover of $A$. In cases where this notation is ambiguous we will use an appropriately descriptive unambiguous version of the definition above.

\begin{definition}
	Let $(X, d)$ be a metric space,  $\mathfrak{F} \subseteq \mathcal{P}(X)$ and $\zeta_s (\cdot) = d(\cdot)^s$, then
	for each $s \in (0,\infty)$ we construct the $s$-dimensional size $\delta$ approximating measures $\mathcal{H}_{\delta}^s$ and the $s$-dimensional Hausdorff Measure, $\mathcal{H}^s$ via the Carath\'{e}odory Construction.
\end{definition}

One should note immediately that rather than constructing one outer-measure we are actually constructing a family of outer-measures parameterized by $s \in [0, \infty).$ This family has the interesting property, which will be shown in \cite{worth}, that each outer-measure provides useful information about a different family of subsets of $X.$
\begin{definition}
	The Hausdorff Dimension (or Hausdorff-Besicovitch Dimension)
	of a set $A$ is the unique $s \in [0,\infty)$ such that 
	$$\mathcal{H}^t (A) =
	\left\{ \begin{array}{c}
		\infty \, \,for \, all  \,\, 0 \leq t < s \\
		0 \,\,\, for \, all  \,\, t > s.
	\end{array}
	\right.$$
\end{definition}

We denote the Hausdorff dimension of a set $A$ by $dim_{\mathcal{H}} (A).$

\section*{Acknowledgements}
M. A. C. Florentino  is  
was financed in part by the Coordena\c{c}\~{a}o de Aperfei\c{c}oamento de Pessoal de N\'{i}vel Superior – Brasil (CAPES) – Finance Code $001.$ 

T. Carvalho is partially supported by S\~{a}o Paulo Research Foundation (FAPESP) grants 2019/10269-3 and 2021/12395-6 and by CNPq-BRAZIL grant 304809/2017-9. 

J. Cassiano is partially supported by FP7-PEOPLE-2012-IRSES-316338.



\end{document}